\documentclass{amsart}

\usepackage{amssymb}
\usepackage{amsmath}
\usepackage{amsfonts}
\usepackage{amstext}
\usepackage{amsthm}

\usepackage{ae,aecompl}
\usepackage[T1]{fontenc}

\usepackage{color}                    % For creating coloured text and background
\usepackage{hyperref}                 % For creating hyperlinks in cross references

\usepackage{verbatim} 		      %Can do multiline comments with \begin{comment}

\newtheorem{theorem}{Theorem}[section]
\newtheorem{lemma}[theorem]{Lemma}
\newtheorem{proposition}[theorem]{Proposition}
\newtheorem{corollary}[theorem]{Corollary}

\theoremstyle{definition}
\newtheorem{remark}[theorem]{Remark}
\newtheorem{definition}[theorem]{Definition}
\newtheorem{example}[theorem]{Example}

\newtheorem*{claim}{Claim}

\numberwithin{equation}{section}

\DeclareMathOperator{\Alg}{Alg}
\DeclareMathOperator{\spa}{span}

\newcommand{\Sc}[1]{\mathcal{#1}}
\newcommand{\F}[1]{\mathfrak{#1}}
\newcommand{\B}[1]{\mathbb{#1}}

\newcommand{\<}{\langle}
\renewcommand{\>}{\rangle}

\title[Finitely Correlated Representations]{Finitely Correlated Representations of Product Systems of $C^*$-Correspondences over $\mathbb{N}^k$}
\author{Adam Hanley Fuller}
\address{Pure Math. Dept., U. Waterloo, Waterloo, ON N2L-3G1, CANADA}
\email{a2fuller@math.uwaterloo.ca}
\date{\today}

\begin{document}

  \begin{abstract}
   We study isometric representations of product systems of correspondences over the semigroup $\mathbb{N}^k$ which are minimal dilations of finite dimensional, fully coisometric representations. We show the existence of a unique minimal cyclic coinvariant subspace for all such representations. The compression of the representation to this subspace is shown to be a complete unitary invariant. For a certain class of graph algebras the nonself-adjoint \textsc{wot}-closed algebra generated by these representations is shown to contain the projection onto the minimal cyclic coinvariant subspace. This class includes free semigroup algebras. This result extends to a class of higher-rank graph algebras which includes higher-rank graphs with a single vertex.
  \end{abstract}
  
\maketitle

  \section{Introduction}

   A $C^*$-correspondence over a $C^*$-algebra $\Sc{A}$ is a Hilbert bimodule with an $\Sc{A}$-valued inner product. The $C^*$-algebras of representations of $C^*$-correspondences were first studied by Pimsner \cite{Pimsner}.\ In a series of papers beginning with \cite{MuhlySolel1}, Muhly and Solel studied representations of $C^*$-corres\-pond\-ences and their algebras. Remarkably they managed to achieve many results from single operator theory in this very general setting. In \cite{MuhlySolel1} they include a dilation theorem which supersedes the classical Sz.-Nagy \cite{Nagy} dilation theorem for contractions and the Frazho-Bunce-Popescu \cite{Frazho, Bunce, Popescu} dilation theorem for row-contractions. In \cite{MuhlySolel2} a Wold decomposition is presented as well as a Beurling-type theorem.

Product systems of $C^*$-correspondences over the semigroup $\B{R}_+$ were introduced by Arveson in \cite{Arveson}.\ The study of product systems over discrete semigroups began with Fowler's work in \cite{Fowler}, where the generalised Cuntz-Pimsner algebra associated to a product system was introduced.\ In recent years there have been several papers considering product systems of $C^*$-correspondences over discrete semigroups, e.g. \cite{Shal, Shal2, Skalski, SkalZach1, Solel2, Solel3}. There has been work on dilation results for representations of product systems generalizing dilation results for commuting contractions. For example, Solel \cite{Solel2} shows the existence of a dilation for contractive representations of product systems over $\B{N}^2$. This result is analogous to the well-known Ando's theorem for two commuting contractions \cite{Ando}. Solel \cite{Solel3} gives necessary and sufficient conditions for a contractive representation of a product system over $\B{N}^k$ to have what is known as a \emph{regular} dilation. This result is analogous to a theorem of Brehmer \cite{Brehmer}. Skalski and Zacharias \cite{SkalZach1} have presented a Wold decomposition for representations of product systems over $\B{N}^k$.

	The generalised Cuntz-Pimsner $C^*$-algebras associated to product systems over the semigroup $\B{N}^k$ are not in general GCR, i.e. they can be NGCR. A theorem due to Glimm \cite[Theorem 2]{Glimm} tells us that NGCR $C^*$-algebras do not have smooth duals, i.e. there is no countable family of Borel functions on the space of unitary equivalence classes of irreducible representations which separates points. It follows that trying to classify all irreducible representations up to unitary equivalence of a generalised Cuntz-Pimsner algebra would be a fruitless task. However, in this paper we find a complete unitary invariant for a certain class of representations: \emph{finitely correlated representations}.

	An isometric representation of a product system of $C^*$-correspondences is finitely correlated if it is the minimal isometric dilation of a finite dimensional representation. We show the existence of a unique minimal cyclic coinvariant subspace for finitely correlated, isometric, fully coisometric representations of product systems over the semigroup $\B{N}^k$. The compression of the representation to this minimal subspace will be the complete unitary invariant. This result generalises the work of Davidson, Kribs and Shpigel \cite{DavKriShp} for the minimal isometric dilation $[S_1,\ldots,S_n]$ of a finite dimensional row-contraction. Indeed, studying row-contractions is equivalent to studying representations of the $C^*$-correspondence $\B{C}^n$ over the $C^*$-algebra $\B{C}$.\ In \cite{DavKriShp}, it is shown that the projection onto the minimal coinvariant subspace is contained in the \textsc{wot}-closed algebra generated by the $S_i$'s. This is an important invariant for free semigroup algebras \cite{DavKatPitts}. We are able to establish this in a number of interesting special cases.

	Finitely correlated representations were first introduced by Bratteli and Jorgensen \cite{BrattJorg1} via finitely correlated states on $\Sc{O}_n$. When $\omega$ is a finitely correlated state on $\Sc{O}_n$, the GNS construction on $\omega$ will give a representation $\pi$ of $\Sc{O}_n$ with the property that $[\pi(s_1),\ldots,\pi(s_n)]$ is a finitely correlated row isometry, where $s_1,\ldots,s_n$ are pairwise orthogonal isometries generating $\Sc{O}_n$. This relates \cite{DavKriShp} with \cite{BrattJorg1}. Similarly, following the work of Skalski and Zacharias \cite{SkalZach2}, we will define what it means for a state on the Cuntz-Pimsner algebra $\Sc{O}_\Lambda$ for finite $k$-graph $\Lambda$ to be finitely correlated.\ Finitely correlated states will give rise to finitely correlated representations of the product system associated to $\Lambda$.

In \cite{DavPitts} Davidson and Pitts classified atomic representations of $\Sc{O}_n$, which include as a special case the permutation representations studied by Bratteli and Jorgensen \cite{BrattJorg2}. If $s_1,\ldots,s_n$ are pairwise orthogonal isometries which generate $\Sc{O}_n$ then a representation $\pi$ of $\Sc{O}_n$ on a Hilbert space $\Sc{H}$ is atomic if there is an orthonormal basis for $\Sc{H}$ which is permuted by each $\pi(s_i)$ up to multiplication by scalars in $\B{T}\cup\{0\}$. There exist finitely correlated atomic representations of $\Sc{O}_n$ \cite{DavPitts}. Atomic representations have been a used in the study of other objects. In \cite{DavKat2} Davidson and Katsoulis show that the $C^*$-envelope of $\F{A}_n\times_{\varphi}\B{Z}^+$ is $\Sc{O}_n\times_{\varphi}\B{Z}$, where $\F{A}_n$ is the noncommutative disc algebra. Finitely correlated atomic representations of $\Sc{O}_n$ are used as a tool to get to this result, see \cite[Theorem 4.4]{DavKat2}. For a general $C^*$-correspondence or product system of $C^*$-correspondences it is not clear what it could mean for a representation to be atomic. Thus the finitely correlated representations presented in this paper are possibly the nearest analogy to finitely correlated atomic representations. In \cite{DavPowYang2, DavYang2} atomic representations of single vertex $k$-graphs have been classified.

In \S 2 we study finitely correlated representations of $C^*$-correspondences. To this end we follow the same program of attack as \cite{DavKriShp}. Many of the proofs follow the same line of argument as the corresponding proofs in \cite{DavKriShp}. When this is the case it is remarked upon. Lemma \ref{lemma41} corresponds to \cite[Lemma 4.1]{DavKriShp}, and is the key technical tool to our analysis in this section. It should be noted that Lemma \ref{lemma41} not just generalises \cite[Lemma 4.1]{DavKriShp}, but the proof presented here greatly simplifies the argument in \cite{DavKriShp}. The main results of this section are summarised in Theorem \ref{summary} and Corollary \ref{complete unitary invariant}.

   Every graph can be associated with a $C^*$-correspondence. Thus results on representations of $C^*$-correspondences also apply to graph algebras. In Section \ref{graph algebras} we apply our results to nonself-adjoint graph algebras. The study of nonself-adjoint graph algebras has received attention in several papers in recent years, e.g. \cite{DavKat, JaeckPower, JuryKribs2, KribsPower1, KribsPower2, Solel1}. We strengthen our results from Section \ref{cstar cor} for the case of an algebra of a finite graph with the \emph{strong double-cycle property}, i.e.\ for finite graphs where every vertex has a path to a vertex which lies on two distinct minimal cycles. We show that the nonself-adjoint \textsc{wot}-closed algebra generated by a finitely correlated, isometric, fully coisometric representation of such a graph contains the projection onto its unique minimal cyclic coinvariant subspace. Aided by the work of Kribs and Power \cite{KribsPower1} and Muhly and Solel \cite{MuhlySolel2} on the algebras of directed graphs we use the same method of proof as in \cite{DavKriShp} to prove this result. This includes the case studied in \cite{DavKriShp}.

   In Section \ref{prod sys section} we prove the prove the main results of the paper (Theorem \ref{prod sys summary} and Corollary \ref{complete unitary invariant 2}) by generalising the results of \S 2 to product systems of $C^*$-correspondences over $\B{N}^k$. Our main tool in this section is Theorem \ref{rank k to rank 1}. A representation of a product system of $C^*$-correspondences provides a representation for each $C^*$-correspondence in the product system. An isometric dilation of a contractive representation of a product system of $C^*$-correspondences gives an isometric dilation of each of the representations of the individual $C^*$-correspondences. Theorem \ref{rank k to rank 1} tells us that if we have a \emph{minimal} isometric dilation of a fully coisometric representation of a product system over $\B{N}^k$, then the dilations of the corresponding representations of certain individual $C^*$-correspondences in the product system will also be minimal. This allows us to deduce the existence of a unique minimal cyclic coinvariant subspace for finitely correlated, isometric, fully coisometric representations of product systems from the $C^*$-correspondence case. In fact, we will show in Theorem \ref{prod sys summary} that the unique minimal cyclic coinvariant subspace for a representation of a product system will be the same unique minimal cyclic coinvariant subspace for a certain $C^*$-correspondence.

   Higher-rank graph algebras were introduced by Kumjian and Pask in \cite{KumPask}. A $k$-graph is, roughly speaking, a set of vertices with $k$ sets of directed edges ($k$ colours), together with a commutation rule between paths of different colours. In the last decade there has been a lot of study on the $C^*$-algebras generated by representations of higher-rank graphs. In more recent years there has been some study on their nonself-adjoint counterparts, see e.g. \cite{KribsPower3, Power}. The case of algebras of higher-rank graphs with a single vertex has proved to be rather interesting. Their study was begun by Kribs and Power \cite{KribsPower3}. Further study has been carried out by Davidson, Power and Yang \cite{Power, DavPowYang1, DavPowYang2, DavYang1, DavYang2, Yang}.

   A $k$-graph can be associated with a product system of $C^*$-correspondences over the discrete semigroup $\B{N}^k$. Thus results on product systems of $C^*$-correspondences over $\B{N}^k$ apply to higher-rank graph algebras. In Section \ref{k graphs} we remark that since certain $1$-graphs contained in a $k$-graph $\Lambda$ share the same unique minimal cyclic coinvariant subspace for a finitely correlated representation, if $\Lambda$ contains a $1$-graph with the strong double-cycle property, then the \textsc{wot}-closed algebra generated by a finitely correlated, isometric, fully coisometric representation will contain the projection onto its minimal cyclic coinvariant subspace. A $k$-graph with only one vertex satisfies this condition.

	In \cite{DavKriShp} the case of non-fully coisometric, finitely correlated row isometries are also studied. The case of finitely correlated representations of product systems of $C^*$-correspondences which are not fully coisometric are not studied in this paper. The reason for this is because, unlike the Frazho-Bunce-Popescu dilation used in \cite{DavKriShp}, dilations of representations of product systems need not be unique if they are not fully coisometric. See \S \ref{prod sys min iso dil section} for a discussion of dilation theorems for representations of product systems of $C^*$-correspondences over $\B{N}^k$.

  \subsection{Acknowledgements}The author would like to thank his advisor, Ken Davidson, for his invaluable advice and support. The author would also like to thank the anonymous reviewer for their careful reading of the manuscript and for providing the author with many helpful notes.

  \section{\texorpdfstring{$C^*$-Correspondences}{C*-Correspondences}}\label{cstar cor}
  \subsection{Preliminaries and Notation}
	We will assume throughout that all $C^*$-alge\-bras are unital and that representations of $C^*$-algebras are unital. The theory will also work for the non-unital case. The details are left to the reader.

  Most of the background on $C^*$-correspondences needed in this paper can be found in the works of Muhly and Solel \cite{MuhlySolel1, MuhlySolel2}. Provided here is a brief summary of the necessary definitions. 

  Let $E$ be a right module over a $C^*$-algebra $\Sc{A}$. An $\Sc{A}$-valued inner product on $E$ is a map $\<\cdot,\cdot\>:E\times E\rightarrow\Sc{A}$ which is conjugate linear in the first variable, linear in the second variable and satisfies
  \begin{enumerate}
   \item $\<\xi,\eta a\>=\<\xi,\eta\>a$
   \item $\<\xi,\eta\>^*=\<\eta,\xi\>$ and
   \item $\<\xi,\xi\>\geq0$ where $\<\xi,\xi\>=0$ if and only if $\xi=0$,
  \end{enumerate}
  for $\xi,\eta\in E$ and $a\in\Sc{A}$. We can define a norm on $E$ by setting $\|\xi\|=\|\<\xi,\xi\>\|^{\frac{1}{2}}$. If $E$ is complete with respect to this norm then it is called a \emph{Hilbert $C^*$-module}. We denote by $\Sc{L}(E)$ the space of all \emph{adjointable} bounded linear functions from $E$ to $E$, i.e. the bounded operators on $E$ with a (necessarily unique) adjoint with respect to the inner product on $E$. The adjointable operators on a Hilbert $C^*$-module form a $C^*$-algebra. For $\xi, \eta\in E$ define $\xi\eta^*\in\Sc{L}(E)$ by  
  \begin{equation*}
   \xi\eta^*(\zeta)=\xi\<\zeta,\eta\>
  \end{equation*}
  for each $\zeta\in E$. Denote by $\Sc{K}(E)$ the closed linear span of $\{\xi\eta^*:\xi,\eta\in E\}$. The space $\Sc{K}(E)$ forms a $C^*$-subalgebra of $\Sc{L}(E)$ referred to as the \emph{compact operators on $E$}. More on Hilbert $C^*$-modules can be found in \cite{Lance}.

  If there is a homomorphism $\varphi$ from $\Sc{A}$ to $\Sc{L}(E)$, then the Hilbert $C^*$-module $E$, together with the left action on $E$ defined by $\varphi$, is a \emph{$C^*$-correspondence over $\Sc{A}$}. If $E$ and $F$ are two $C^*$-correspondences over $\Sc{A}$ we will write $\varphi_E$ and $\varphi_F$ to describe the left action of $\Sc{A}$ on $E$ and $F$ respectively. With that said, when there is little chance of confusion we will write $a\xi$ in place of $\varphi(a)\xi$.
  
  Suppose $E$ and $F$ are two $C^*$-correspondences over a $C^*$-algebra $\Sc{A}$. We define the following $\Sc{A}$-valued inner product on the algebraic tensor product $E\otimes_A F$, of $E$ and $F$: for $\xi_1, \xi_2$ in $E$ and $\eta_1, \eta_2$ in $F$ we let
  \begin{equation*}
   \<\xi_1\otimes\eta_1,\xi_2\otimes\eta_2\>=\<\eta_1,\varphi_F(\<\xi_1,\xi_2\>)\eta_2\>.
  \end{equation*}
  Taking the Hausdorff completion of $E\otimes_A F$ with respect to this inner product gives us the \emph{interior tensor product} of $E$ and $F$ denoted $E\otimes F$. This is the only tensor product of $C^*$-correspondences that we will use in this paper so we will omit the word ``interior'' and merely say we are taking the tensor product of $C^*$-correspondences. When taking the tensor product of a $C^*$-correspondence $E$ with itself we will write $E^2$ in place of $E\otimes E$, and similarly we will write $E^n$ in place of the $n$-fold tensor product of $E$ with itself. We will also set $E^0=\Sc{A}$.

  The Fock space $\Sc{F}(E)$ is defined to be the $C^*$-correspondence 
  \begin{equation*}
   \Sc{F}(E)=\sideset{}{^\oplus}\sum_{n\geq 0}E^n.
  \end{equation*}
  The left action of $\Sc{A}$ on $\Sc{F}(E)$ is denote by $\varphi_\infty$ and defined by
  \begin{equation*}
   \varphi_\infty(a)\xi_1\otimes\ldots\otimes\xi_n=(a\xi_1)\otimes\ldots\otimes\xi_n.
  \end{equation*}
  We define \emph{creation operators} $T_\xi$ in $\Sc{L}(\Sc{F}(E))$ for $\xi\in E$ by
  \begin{equation*}
   T_\xi(\xi_1\otimes\ldots\otimes\xi_n)=\xi\otimes\xi_1\otimes\ldots\otimes\xi_n\in E^{n+1}
  \end{equation*}
  for $\xi_1\otimes\ldots\otimes\xi_n\in E^n$. The norm closed algebra in $\Sc{L}(\Sc{F}(E))$ generated by 
	\begin{equation*}
		\{T_\xi, \varphi_\infty(a):\xi\in E, a\in\Sc{A}\} 
	\end{equation*}
	is denoted by $\Sc{T}_+(E)$ and called the \emph{tensor algebra over $E$}. The $C^*$-algebra generated by $\Sc{T}_	+(E)$ is denoted $\Sc{T}(E)$ and called the \emph{Toeplitz algebra over $E$}.

  A \emph{completely contractive covariant representation} $(A,\sigma)$ of a $C^*$-correspondence $E$ over $\Sc{A}$ on a Hilbert space $\Sc{H}$ is a completely contractive linear map $A$ from $E$ to $\Sc{B}(\Sc{H})$ and a unital, non-degenerate representation $\sigma$ of $\Sc{A}$ on $\Sc{H}$ which satisfy the following covariant property:
  \begin{equation*}
    A(a\xi b)=\sigma(a)A(\xi)\sigma(b)
  \end{equation*}
  for $a,b\in\Sc{A}$ and $\xi\in E$. We will abbreviate \emph{completely contractive covariant representation} to merely \emph{representation}, as these will be the only representations of $C^*$-correspondences we will consider. A representation $(A,\sigma)$ is called \emph{isometric} if it satisfies
  \begin{equation*}
   A(\xi)^*A(\eta)=\sigma(\<\xi,\eta\>).
  \end{equation*}
  Why this is called isometric will become clear presently.

  If $\sigma$ is a representation of $\Sc{A}$ on a Hilbert space $\Sc{H}$ and $E$ is a $C^*$-correspondence over $\Sc{A}$, then we can form a Hilbert space $E\otimes_\sigma\Sc{H}$ by taking the algebraic tensor product of $E$ and $\Sc{H}$ and taking the Hausdorff completion with respect to the inner product defined by
  \begin{equation*}
   \<\xi_1\otimes h_1,\xi_2\otimes h_2\>=\<h_1,\sigma(\<\xi_1,\xi_2\>)h_2\>
  \end{equation*}
  for $\xi_1,\xi_2\in E$ and $h_1,h_2\in\Sc{H}$. We will write $E\otimes\Sc{H}$ in place of $E\otimes_\sigma\Sc{H}$ when it is understood which representation we are talking about. We can \emph{induce} $\sigma$ to a representation $\sigma^E$ of $\Sc{L}(E)$ on $E\otimes\Sc{H}$. This is defined by
  \begin{equation*}
   \sigma^E(T)(\xi\otimes h)=(T\xi)\otimes h
  \end{equation*}
  for $T\in\Sc{L}(E), \xi\in E$ and $h\in\Sc{H}$. In particular we can induce $\sigma$ to $\sigma^{\Sc{F}(E)}$. We define an isometric representation $(V,\rho)$ of $E$ on $\Sc{F}(E)\otimes\Sc{H}$ by
  \begin{equation*}
   \rho(a)=\sigma^{\Sc{F}(E)}\circ\varphi_\infty(a)
  \end{equation*}
  for each $a\in\Sc{A}$ and
  \begin{equation*}
   V(\xi)=\sigma^{\Sc{F}(E)}(T_\xi)
  \end{equation*}
  for each $\xi\in E$. We call $(V,\rho)$ the representation of $E$ \emph{induced by $\sigma$}.

  If $(A,\sigma)$ is a representation of $E$ on $\Sc{H}$, then we define the operator $\tilde{A}$ from $ E\otimes_\sigma\Sc{H}$ to $\Sc{H}$ by
  \begin{equation*}
   \tilde{A}(\xi\otimes h)=A(\xi)h.
  \end{equation*}
  This operator was introduced by Muhly and Solel in \cite{MuhlySolel1}, where they show that $\tilde{A}$ is a contraction. Furthermore, they show that $\tilde{A}$ is an isometry if and only if $(A,\sigma)$ is an isometric representation. A representation is called \emph{fully coisometric} when $\tilde{A}$ is a coisometry. We write $\tilde{A}_n$ for the operator from $E^n\otimes_\sigma\Sc{H}$ to $\Sc{H}$ defined by
  \begin{equation*}
   \tilde{A}_n(\xi_1\otimes\ldots\otimes\xi_n\otimes h)=A(\xi_1)\ldots A(\xi_n)h.
  \end{equation*}
  Note also that $\sigma(a)\tilde{A}=\tilde{A}\sigma^E(\varphi(a))$.

  If $\sigma$ is a representation of $\Sc{A}$ on $\Sc{H}$ and $X$ is in the commutant of $\sigma({\Sc{A}})$, then we can define a bounded operator $I\otimes X$ on $E\otimes\Sc{H}$ by
  \begin{equation*}
   (I\otimes X)(\xi\otimes h)=\xi\otimes Xh.
  \end{equation*}
  It is readily verifiable that $I\otimes X$ is a bounded operator and that $\|I\otimes X\|\leq\|X\|$. In particular if $\Sc{M}$ is a subspace of $\Sc{H}$ with $P_{\Sc{M}}\in\sigma(\Sc{A})'$ then $I\otimes P_{\Sc{M}}$ is a projection in $\Sc{B}(E\otimes\Sc{H})$. Thus $E\otimes\Sc{H}$ decomposes into a direct sum $E\otimes\Sc{H}=(E\otimes\Sc{M})\oplus(E\otimes\Sc{M}^\perp)$.

 Let $(S,\rho)$ be a representation of a $C^*$-correspondence $E$ on a Hilbert space $\Sc{H}$. We denote by $I$ be the identity in $\Sc{B}(H)$. We call the weak-operator topology closed algebra
	\begin{equation*}
		\F{S}=\Alg\{I,\ S(\xi),\ \rho(a):\xi\in E,\ a\in \overline{\Sc{A}\}}^{\textsc{wot}}
	\end{equation*}
the \emph{unital \textsc{wot}-closed algebra generated by the representation $(S,\rho)$.}

  \subsection{Minimal Isometric Dilations}
  \begin{definition}
  Let $E$ be a $C^*$-correspondence over a $C^*$-algebra $\Sc{A}$ and let $(A,\sigma)$ be a representation of $E$ on a Hilbert space $\Sc{V}$. A representation $(S,\rho)$ of $E$ on $\Sc{H}$ is a \emph{dilation} of $(A,\sigma)$ if $\Sc{V}\subseteq\Sc{H}$ and
  \begin{enumerate}
   \item $\Sc{V}$ reduces $\rho$ and $\rho(a)|_{\Sc{V}}=\sigma(a)$ for all $a\in\Sc{A}$.
   \item $\Sc{V}^\perp$ is invariant under $S(\xi)$ for all $\xi\in E$
   \item $P_{\Sc{V}}S(\xi)|_{\Sc{V}}=A(\xi)$ for all $\xi\in E$.
  \end{enumerate}
  A dilation $(S,\rho)$ of $(A,\sigma)$ is an \emph{isometric dilation} if $(S,\rho)$ is an isometric representation. A dilation $(S,\rho)$ of $(A,\sigma)$ on $\Sc{H}$ is called \emph{minimal} if 
   \begin{equation*}
    \Sc{H}=\bigvee_{n\geq0}\tilde{S}_n(E^n\otimes\Sc{V}).
   \end{equation*}
  \end{definition}
  
  \begin{theorem}[Muhly and Solel \cite{MuhlySolel1}]\label{cstar dilation}
   If $(A,\sigma)$ is a contractive representation of a $C^*$-correspondence $E$ on a Hilbert space $\Sc{V}$, then $(A,\sigma)$ has an isometric dilation $(S,\rho)$. Further, we can choose $(S,\rho)$ to be minimal; and  the minimal isometric dilation of $(A,\sigma)$ is unique up to a unitary equivalence which fixes $\Sc{V}$.
  \end{theorem}

  The following lemma uses a standard argument in dilation theory.

  \begin{lemma}\label{coiso}
   If $(A, \sigma)$ is a representation of a $C^*$-correspondence $E$ on a Hilbert space $\Sc{V}$ and $(S,\rho)$ is its minimal isometric dilation on $\Sc{H}$, then $(S,\rho)$ is fully coisometric if and only if $(A, \sigma)$ is fully coisometric.
  \end{lemma}

  \begin{proof}
   Clearly, if $\tilde{S}\tilde{S}^*=I_\Sc{H}$ then for $v\in\Sc{V}$, $\tilde{A}\tilde{A}^*v=P_{\Sc{V}}\tilde{S}\tilde{S}^*v=P_\Sc{V}v=v$, and so $\tilde{A}$ is a coisometry.
  
   Conversely, suppose that $\tilde{A}$ is a coisometry. Let $\Sc{M}=(I-\tilde{S}\tilde{S}^*)\Sc{H}$. It is not hard to see that $\Sc{M}$ is a $\F{S}^*$-invariant subspace, where $\F{S}$ is the unital \textsc{wot}-closed algebra generated by the representation $(S,\rho)$. Also since $\tilde{A}\tilde{A}^*=I_{\Sc{V}}$ we have that $P_\Sc{V}\tilde{S}\tilde{S}^*|_\Sc{V}=I_\Sc{V}$, hence $\Sc{M}$ is a $\F{S}^*$-invariant space orthogonal to $\Sc{V}$. But, since our dilation is minimal the only $\F{S}^*$-invariant subspace orthogonal to $\Sc{V}$ is the zero space. Therefore $\Sc{M}=\{0\}$.
  \end{proof}

  The following two results have been proved in \cite{DavKriShp} for the case when $E=\B{C}^n$ (where $2\leq n\leq\infty$) and $\Sc{A}=\B{C}$. We follow much the same line of proof as found there.

  \begin{lemma}\label{wanderingspace}
	Let $(A,\sigma)$ be a representation of a $C^*$-correspondence $E$ on a Hilbert space $\Sc{V}$, and let $(S,\rho)$ be the unique minimal isometric dilation of $(A,\sigma)$ on a Hilbert space $\Sc{H}$. Let $\Sc{W}=(\Sc{V}+\tilde{S}(E\otimes\Sc{V}))\ominus\Sc{V}$. Then $\Sc{W}$ is a $\rho$-reducing subspace and $\Sc{V}^{\perp}$ is isometrically isomorphic to $\Sc{F}(E)\otimes\Sc{W}$. Furthermore, the representation of $E$ obtained by restricting $(S,\rho)$ to $\Sc{V}^{\perp}$ is the representation induced by $\rho(\cdot)|_{\Sc{W}}$.
  \end{lemma}

  \begin{proof}
   First note that $\Sc{W}$ is $\rho$-reducing. This follows since $\Sc{V}$ is $\rho$-reducing and hence so is $\Sc{V}^\perp$ and $\rho(a)S(\xi)\Sc{V}=S(a\xi)\Sc{V}$ for each $a\in\Sc{A}$ and $\xi\in E$. 

   The subspace $\Sc{V}^\perp$ is invariant under $S(\xi)$ for each $\xi\in E$. So for any $n$ and $\xi_1,\ldots,\xi_n\in E$, the space $S(\xi_1)\ldots S(\xi_n)\Sc{W}$ is orthogonal to $\Sc{V}$. It follows that if $n\geq1$, then $S(\xi_1)\ldots S(\xi_n)\Sc{W}$ is orthogonal to $S(\xi)\Sc{V}$ for all $\xi\in E$. Therefore $S(\xi_1)\ldots S(\xi_n)\Sc{W}$ is orthogonal to $\Sc{V}+\tilde{S}(E\otimes\Sc{V})$, which contains $\Sc{W}$.

   Also note that if $\eta_1,\ldots,\eta_m$ are in $E$, with $m<n$ and $w_1$ and $w_2$ in $\Sc{W}$ then
   \begin{multline*}
    \<S(\xi_1)\ldots S(\xi_n)w_1,S(\eta_1)\ldots S(\eta_m)w_2\>\\=\<\rho(\<\eta_1\otimes\ldots\otimes\eta_m,\xi_1\otimes\ldots\otimes\xi_m\>)S(\xi_{m+1})\ldots S(\xi_n)w_1,w_2\>=0.
   \end{multline*}
   By minimality we have that
   \begin{align*}
    \Sc{V}^\perp&=\sideset{}{^\oplus}\sum_{n\geq0}\sum_{\xi_1,\ldots,\xi_n\in E}S(\xi_1)\ldots S(\xi_n)\Sc{W}\\
    &=\sideset{}{^\oplus}\sum_{n\geq0}\tilde{S^n}(E^n\otimes\Sc{W})\\
    &\simeq\Sc{F}(E)\otimes\Sc{W}.\qedhere
   \end{align*}
  \end{proof}

  \begin{remark}
   When $(A,\sigma)$ is a fully coisometric representation of a $C^*$-correspon\-dence $E$ on a Hilbert space $\Sc{V}$, we have that $\Sc{V}=\tilde{S}\tilde{S}^*\Sc{V}=\tilde{S}\tilde{A}^*\Sc{V}\subseteq\tilde{S}(E\otimes \Sc{V})$. Hence, when $(A,\sigma)$ is fully coisometric the space $\Sc{W}$ in Lemma \ref{wanderingspace} is simply $\Sc{W}=\tilde{S}(E\otimes \Sc{V})\ominus\Sc{V}$.
  \end{remark}

  \begin{lemma}\label{reducingspace}
	Let $(A,\sigma)$ be a representation of a $C^*$-correspondence $E$ on a Hilbert space $\Sc{V}$, and let $(S,\rho)$ be the unique minimal isometric dilation of $(A,\sigma)$ on a Hilbert space $\Sc{H}$. Let $\F{A}$ be the \textsc{wot}-closed unital algebra generated by $(A,\sigma)$ and let $\F{S}$ be the \textsc{wot}-closed unital algebra generated by $(S,\rho)$.
   Suppose $\Sc{V}_1$ is an $\F{A}^*$-invariant subspace of $\Sc{V}$. Then $\Sc{H}_1=\F{S}[\Sc{V}_1]$ reduces $\F{S}$.

   If $\Sc{V}_1$ and $\Sc{V}_2$ are orthogonal $\F{A}^*$-invariant subspaces, then $\Sc{H}_j=\F{S}[\Sc{V}_j]$ for $j=1,2$ are mutually orthogonal.

   If $\Sc{V}=\Sc{V}_1\oplus\Sc{V}_2$, then $\Sc{H}=\Sc{H}_1\oplus\Sc{H}_2$ and $\Sc{H}_j\cap\Sc{V}=\Sc{V}_j$ for $j=1,2$.
  \end{lemma}

  \begin{proof}
   Note that for any $a\in\Sc{A}$, $\rho(a)\Sc{V}_1=\sigma(a)\Sc{V}_1\subseteq\Sc{V}_1$. Also for any $\xi\in E$, $S(\xi)^*\Sc{V}_1=A(\xi)^*\Sc{V}_1\subseteq\Sc{V}_1$. Hence $\Sc{V}_1$ is $\F{S}^*$-invariant. Now $\Sc{H}_1$ is spanned by vectors of the form $S(\xi_1)\ldots S(\xi_n)v$, where $\xi_1,\ldots,\xi_n\in E$ and $v\in\Sc{V}_1$. If $n\geq2$ then for any $\xi\in E$,
   \begin{equation*}
    S(\xi)^*S(\xi_1)\ldots S(\xi_n)v=S(\<\xi,\xi_1\>\xi_2)\ldots S(\xi_n)v\in\Sc{H}_1.
   \end{equation*}
   If $n=1$ we have $S(\xi)^*S(\xi_1)v=\rho(\<\xi,\xi_1\>)v=\sigma(\<\xi,\xi_1\>)v\in\Sc{H}_1$. Hence $\Sc{H}_1$ reduces $\F{S}$.

   Now suppose $\Sc{V}_1$ and $\Sc{V}_2$ are orthogonal $\F{A}^*$-invariant subspaces. Take $v_1\in\Sc{V}_1$, $v_2\in\Sc{V}_2$ and $\xi_1,\ldots,\xi_n,\eta_1,\ldots,\eta_m$ be in $E$. Suppose $n\geq m$ then
   \begin{multline*}
    \<S(\xi_1)\ldots S(\xi_n)v_1,S(\eta_1)\ldots S(\eta_m)v_2\>\\=\<v_1,S(\xi_n)^*\ldots S(\xi_{m+1})^*\rho(\<\xi_m,\eta_m\>\ldots\<\xi_1,\eta_1\>)v_2\>=0.
   \end{multline*}
   It follows that $\Sc{H}_1$ and $\Sc{H}_2$ are orthogonal.

   If $\Sc{V}=\Sc{V}_1\oplus\Sc{V}_2$ then, since $\Sc{H}_1$ contains $\Sc{V}_1$ and is orthogonal to $\Sc{V}_2$, $\Sc{H}_1\cap\Sc{V}=\Sc{V}_1$. Finally, $\Sc{H}_1\oplus\Sc{H}_2$ is an $\F{S}$-reducing subspace containing $\Sc{V}$, so it is all of $\Sc{H}$ by the minimality of the dilation.
  \end{proof}

  Given an isometric representation $(S,\rho)$ of a $C^*$-correspondence $E$ on $\Sc{H}$ with corresponding unital \textsc{wot}-closed algebra $\F{S}$ which is the minimal isometric dilation of a representation $(A,\sigma)$ on $\Sc{V}\subseteq\Sc{H}$, Lemma \ref{reducingspace} shows that $\F{S}^*$-invariant subspaces of $\Sc{V}$ give rise to $\F{S}$-reducing subspaces of $\Sc{H}$. In Corollary \ref{detbyVcor} we give a weak converse of this: that $\F{S}$-reducing subspaces in $\Sc{H}$ are uniquely determined by their projections onto $\Sc{V}$. This follows from the following more general result.

  \begin{lemma}\label{detbyV}
	Let $(A,\sigma)$ be a representation of a $C^*$-correspondence $E$ on a Hilbert space $\Sc{V}$, and let $(S,\rho)$ be the unique minimal isometric dilation of $(A,\sigma)$ on a Hilbert space $\Sc{H}$. Let $\F{S}$ be the unital \textsc{wot}-closed algebra generated by $(S,\rho)$. Suppose $B$ is a normal operator in $\Sc{B}(\Sc{H})$ such that the range of $B$ is contained in $\Sc{V}^\perp$ and $B$ is in $C^*(S(E),\rho(\Sc{A}))'$, the commutant of the $C^*$-algebra generated by $S(E)$ and $\rho(\Sc{A})$. Then $B=0$. 
  \end{lemma}
 
  \begin{proof}
    Suppose that $B$ is non-zero. Take any $\delta$ such that $0<\delta<\|B\|$  and let $D_\delta$ be the open disc of    radius $\delta$ about $0$. Let $Q$ be the spectral projection $Q=E_B(\text{spec}(B)\backslash D_\delta)$, where spec$(B)$ denotes the spectrum of $B$. Then $Q\in W^*(B)\subseteq C^*(S(E),\rho(\Sc{A}))'$ and $Q\Sc{H}$ is orthogonal to $\Sc{V}$. In particular $Q\Sc{H}$ is a non-zero $\F{S}^*$-invariant space orthogonal to $\Sc{V}$. But no such space can exist since our dilation is minimal.
  \end{proof}

  \begin{corollary}\label{detbyVcor}
   Suppose $\Sc{M}$ and $\Sc{N}$ are two $\F{S}$-reducing subspaces of $\Sc{H}$ and the compressions of $P_{\Sc{M}}$ and $P_{\Sc{N}}$ to $\Sc{V}$ are equal, i.e. $P_\Sc{V} P_\Sc{M} P_\Sc{V}=P_\Sc{V} P_\Sc{N} P_\Sc{V}$. Then $\Sc{M}=\Sc{N}$.
  \end{corollary}

  \begin{proof}
   Let $\Sc{M}$ and $\Sc{N}$ be two $\F{S}$-reducing subspaces with $P_\Sc{V} P_\Sc{M} P_\Sc{V}=P_\Sc{V} P_\Sc{N} P_\Sc{V}$.
   Elements of the form $S(\xi_1)\ldots S(\xi_n)v$, with $v\in\Sc{V}$ and $\xi_1,\ldots,\xi_n\in E$, span a dense subset of $\Sc{H}$ and
   \begin{align*}
    (P_\Sc{M}-P_\Sc{N})S(\xi_1)\ldots S(\xi_n)v&=S(\xi_1)\ldots S(\xi_n)(P_\Sc{M}-P_\Sc{N})v\\&=S(\xi_1)\ldots S(\xi_n)P_{\Sc{V}^\perp}(P_\Sc{M}-P_\Sc{N})v\in\Sc{V}^\perp.
   \end{align*}
   It follows that the range of $P_\Sc{M}-P_\Sc{N}$ lies in $\Sc{V}^\perp$. Hence, by Lemma \ref{detbyV} $P_\Sc{M}-P_\Sc{N}=0$ and $\Sc{M}=\Sc{N}$.
  \end{proof}

  \subsection{Finitely Correlated Representations}
  \begin{definition}
    An isometric representation $(S,\rho)$ of a $C^*$-correspondence $E$ on a Hilbert space $\Sc{H}$ is called \emph{finitely correlated} if $(S,\rho)$ is the minimal isometric dilation of a representation $(A,\sigma)$ on a non-zero finite dimensional Hilbert space $\Sc{V}\subseteq\Sc{H}$.

	In particular, if $\F{S}$ is the unital \textsc{wot}-closed algebra generated by $(S,\rho)$, then $(S,\rho)$ is finitely correlated if there is a finite dimensional $\F{S}^*$-invariant subspace $\Sc{V}$ of $\Sc{H}$ such that $(S,\rho)$ is the minimal isometric dilation of the representation $(P_{\Sc{V}}S(\cdot)|_{\Sc{V}},\rho(\cdot)|_{\Sc{V}})$.
  \end{definition}

	\begin{remark}\label{finitely representable}
		It should be noted that not all $C^*$-algebras can be represented non-trivially on a finite dimensional Hilbert spaces, e.g. if $\Sc{A}$ is a properly infinite $C^*$-algebra then there are no non-zero finite dimensional representations of $\Sc{A}$ since $\Sc{A}$ contains isometries with pairwise orthogonal ranges. Likewise, any simple infinite dimensional $C^*$-algebra has no finite dimensional representations.

%	Note also that if $(A,\sigma)$ is a fully coisometric representation of a $C^*$-correspondence $E$ on a non-zero Hilbert space $\Sc{V}$, then $\sigma$ is necessarily not the zero representation. If $\sigma$ were zero everywhere then $E\otimes_\sigma\Sc{V}$ would be the zero space. Thus the map $\tilde{A}^*$ from $\Sc{V}$ to $E\otimes_\sigma\Sc{V}$ would not be isometric, and so the representation $(A,\sigma)$ would not be fully coisometric.

	In this section we are concerned with finitely correlated fully coisometric representations. If we assume that a $C^*$-correspondence $E$ over a $C^*$-algebra $\Sc{A}$ has a fully coisometric representation then we are assuming that there are non-zero representations of $\Sc{A}$ on finite-dimensional Hilbert spaces. Under this assumption there are still a wide range of $C^*$-correspondences which can be studied, e.g. the following example and the $C^*$-correspondences associated to graphs in \S \ref{examples}.
	\end{remark}

  \begin{example}
   The case when $\Sc{A}=\B{C}$ and $E=\B{C}^n$ has been studied previously in \cite{DavKriShp}. A representation of $E$ on a finite dimensional space $\Sc{V}$ is simply a row-contraction $A=[A_1,\ldots,A_n]$ from $\Sc{V}^{(n)}$ to $\Sc{V}$. The representation is fully-coisometric when $A$ is \emph{defect free}, i.e.
   \begin{equation*}
    \sum_{i=1}^n A_iA_i^*=I_\Sc{V}.
   \end{equation*}
   The dilation of $A$ will be the Frazho-Bunce-Popescu dilation of $A$ to a row-isometry $S=[S_1,\ldots,S_n]$. The dilation $S$ will be defect free as $A$ is. These representations can alternatively be viewed as representations of a graph with $1$ vertex and $n$ edges, see \S\ref{graph algebras}.
  \end{example}

  Let $(S,\rho)$ be a fully coisometric, finitely correlated representation on $\Sc{H}$ of the $C^*$-correspondence $E$ over the $C^*$-algebra $\Sc{A}$, and let $\F{S}$ be the unital \textsc{wot}-closed algebra generated by $(S,\rho)$. A key tool in the analysis in \cite{DavKriShp} is that every non-zero $\F{S}^*$-invariant subspace of $\Sc{H}$ has non-trivial intersection with $\Sc{V}$ (\cite[Lemma 4.1]{DavKriShp}), for the case $\Sc{A}=\B{C}$ and $E=\B{C}^n$. The main idea of the proof is that, because the representation is fully coisometric and the unit ball in $\Sc{V}$ is compact, one can ``pull back'' any non-zero element of $\Sc{H}$ with elements in $\F{S}^*$ to $\Sc{V}$, without the norm going to zero. However, the proof in \cite{DavKriShp} that the norm does not go to zero is quite complicated. We prove the analogous result for more general $C^*$-correspondences than those studied in \cite{DavKriShp} below. The proof presented below simplifies the approach in \cite{DavKriShp} by ``pulling back'' not in $\Sc{H}$ but in $\Sc{F}(E)\otimes\Sc{H}$, making use of Muhly and Solel's \ $\tilde{}$\  operators.

  \begin{lemma}\label{lemma41}
	Let $(S,\rho)$ be a finitely correlated, fully coisometric representation of a $C^*$-correspondence $E$ on $\Sc{H}$. Let $\F{S}$ be the unital \textsc{wot}-closed algebra generated by $(S,\rho)$ and let $\Sc{V}$ be a finite dimensional $\F{S}^*$-invariant subspace of $\Sc{H}$ such that $(S,\rho)$ is the minimal isometric dilation of the representation $(P_{\Sc{V}}S(\cdot)|_{\Sc{V}},\rho(\cdot)|_{\Sc{V}})$.

   Then if $\Sc{M}$ is a non-zero, $\F{S}^*$-invariant subspace of $\Sc{H}$, the subspace $\Sc{M}\cap\Sc{V}$ is non-trivial.
  \end{lemma}

  \begin{proof}
    Let $\mu=\|P_\Sc{V} P_\Sc{M}\|$. If $\mu=1$ then for each $n$ there is a unit vector $h_n\in\Sc{M}$ such that $\|P_{\Sc{V}^\perp} h_n\|<\frac{1}{n}$. Let $v_n=P_\Sc{V} h_n$. We have that $(v_n)_n$ is a sequence in the unit ball of $\Sc{V}$ therefore it has a convergent subsequence $(v_{n_i})_{n_i}$. Let $v_0$ be the limit of $(v_{n_i})_{n_i}$. We have then that
    \begin{equation*}
     \|h_{n_i}-v_0\|\leq\|h_{n_i}-v_{n_i}\|+\|v_{n_i}-v_0\|\rightarrow0,
    \end{equation*}
    as $n_i\rightarrow\infty$ and so the subsequence $(h_{n_i})_{n_i}$ converges to $v_0$. Therefore $v_0$ is a non-zero vector in $\Sc{M}\cap\Sc{V}$. Thus showing that $\mu=1$ will prove the lemma.

    Let $h$ be a unit vector in $\Sc{M}$. Since our dilation is minimal there is a sequence $(k_n)_n$ converging to $h$ where each $k_n$ is of the form
    \begin{equation*}
     k_n=\sum_{i=1}^{N_n}S(\xi_{n,i,1})\ldots S(\xi_{n,i,w_{n,i}})v_{n,i}
    \end{equation*}
    with $\xi_{n,i,j}\in E$ and $v_{n,i}\in\Sc{V}$. Without loss of generality we can assume that $\|k_n\|=1$ for each $n$.

    If we let $M_n=\max\{w_{n,i}:1\leq i\leq N_n\}$ for each $n$, then for any $\xi_1,\ldots,\xi_{M_n}\in E$ we have
    \begin{equation*}
     S(\xi_1)^*\ldots S(\xi_{M_n})^*k_n\in\Sc{V}.
    \end{equation*}
    It follows that $\tilde{S}_{M_n}^*k_n\in E^{M_n}\otimes\Sc{V}$. Note that $\tilde{S}_{M_n}$ is a coisometry so $\|\tilde{S}_{M_n}^*k_n\|=1$. We also have that $\tilde{S}_{M_n}^* h\in E^{M_n}\otimes\Sc{M}$ and $\|\tilde{S}_{M_n}^* h\|=1$.

    Let $u_n=\tilde{S}_{M_n}^*k_n$ and $h_n=\tilde{S}_{M_n}^* h$. We have that
    \begin{equation*}
     \|u_n-h_n\|\rightarrow0
    \end{equation*}
    as $n\rightarrow0$. If $\mu<1$ we can choose $\varepsilon>0$ such that $1-\varepsilon\geq\mu$ and take $n$ large enough so that
    \begin{equation*}
     \|h_n-u_n\|^2=\|h_n\|^2+\|u_n\|^2-2 Re\<h_n,u_n\><2\varepsilon.
    \end{equation*}
    It follows that
    \begin{align*}
     1-\varepsilon&< Re\<h_n,u_n\>\\
     &\leq\|(I_{E_{M_n}}\otimes P_\Sc{V}) h_n\|\|u_n\|,
    \end{align*}	
    with last inequality being the Cauchy-Schwarz inequality. So our choice of $\varepsilon$ tells us that $\mu<\|(I_{E_{M_n}}\otimes P_{\Sc{V}}P_{\Sc{M}})\|\leq\|P_{\Sc{V}}P_{\Sc{M}}\|$. This is a contradiction. Thus $\mu=1$.
  \end{proof}

  \begin{proposition}\label{reducingspacecor}
	Let $(A,\sigma)$ be a representation of a $C^*$-correspondence $E$ on a finite dimensional Hilbert space $\Sc{V}$, and let $(S,\rho)$ be the unique minimal isometric dilation of $(A,\sigma)$ on a Hilbert space $\Sc{H}$. Let $\F{A}$ be the unital algebra generated by the representation $(A,\sigma)$ and let $\F{S}$ be the unital \textsc{wot}-closed algebra generated by $(S,\rho)$.

   If $\Sc{V}_1$ is an $\F{A}^*$-invariant subspace of $\Sc{V}$ and $\Sc{H}_1=\F{S}[\Sc{V}_1]$. Then $\Sc{H}_1\cap\Sc{V}=\F{A}[\Sc{V}_1]$.
  \end{proposition}

  \begin{proof}
   If $w\in\Sc{V}\ominus\F{A}[\Sc{V}_1]$ then $\F{A}^*w$ is an $\F{A}^*$-invariant space orthogonal to $\Sc{V}_1$, hence by Lemma \ref{reducingspace} $\F{S}[\F{A}^*w]\subseteq\Sc{H}_1^\perp$. Therefore  $\Sc{H}_1\cap\Sc{V}\subseteq\F{A}[\Sc{V}_1]$.

   If $w\in\Sc{H}_1^\perp\cap\Sc{V}$ then for any $A\in\F{A}$ and $v\in\Sc{V}_1$ then we have that $0=\<A^*w,v\>=\<w,Av\>$. Hence $\F{A}[\Sc{V}_1]\subseteq\Sc{H}_1\cap\Sc{V}$.
  \end{proof}

  \begin{corollary}
    If $\Sc{M}$ is a $\F{S}$-reducing subspace then $\Sc{M}=\F{S}[\Sc{M}\cap\Sc{V}]$.
  \end{corollary}

  \begin{proof}
    $\Sc{M}$ is a $\F{S}$-reducing subspace and, by Lemma \ref{reducingspace}, $\F{S}[\Sc{M}\cap\Sc{V}]$ is a $\F{S}$-reducing subspace. Hence $\Sc{M}\ominus\F{S}[\Sc{M}\cap\Sc{V}]$ is $\F{S}$-reducing. If $\Sc{M}\ominus\F{S}[\Sc{M}\cap\Sc{V}]$ is non-zero then by Lemma \ref{lemma41}, $\Sc{M}\ominus\F{S}[\Sc{M}\cap\Sc{V}]$ has non-zero intersection with $\Sc{V}$. This yields a contradiction as the intersection will be orthogonal to $\Sc{M}\cap\Sc{V}$.
  \end{proof}

  \begin{corollary}\label{lemma41corollary2}
   If $\F{A}=\Sc{B}(\Sc{V})$ then every $\F{S}^*$-invariant subspace of $\Sc{H}$ contains $\Sc{V}$.
  \end{corollary}

  \begin{proof}
   Suppose $\Sc{M}$ is a non-zero $\F{S}$-reducing subspace. Then $\Sc{M}\cap\Sc{V}$ is a non-zero $\F{S}^*$-invariant, and hence $\F{A}^*$-invariant, subspace of $\Sc{V}$. Hence $\Sc{M}\cap\Sc{V}=\Sc{V}$.
  \end{proof}

  \begin{corollary}\label{uniqueminimal}
   If $\Sc{V}_1$ and $\Sc{V}_2$ are minimal $\F{A}^*$-invariant subspaces of $\Sc{V}$ such that $\F{S}[\Sc{V}_1]=\F{S}[\Sc{V}_2]$, then $\Sc{V}_1=\Sc{V}_2$.
  \end{corollary}

  \begin{proof}
   Let $\Sc{H}'=\F{S}[\Sc{V}_1]=\F{S}[\Sc{V}_2]$. Define representations $(B,\sigma_1)$ and $(C,\sigma_2)$ of $E$ on $\Sc{V}_1$ and $\Sc{V}_2$ respectively by
   \begin{equation*}
    B(\xi)=P_{\Sc{V}_1}A(\xi)|_{\Sc{V}_1}\text{ and }C(\xi)=P_{\Sc{V}_2}A(\xi)|_{\Sc{V}_2}
   \end{equation*}
   for all $\xi\in E$, and
   \begin{equation*}
    \sigma_i(a)=\sigma(a)|_{\Sc{V}_i}
   \end{equation*}
   for all $a\in\Sc{A}$, $i=1,2$. The representations $(B,\sigma_1)$ and $(C,\sigma_2)$ share a unique minimal isometric dilation $(S(\cdot)|_{\Sc{H}'},\sigma(\cdot)|_{\Sc{H}'})$. By Corollary \ref{lemma41corollary2}, any $\F{S}^*$-invariant subspace of $\Sc{H}'$ contains both $\Sc{V}_1$ and $\Sc{V}_2$. In particular $\Sc{V}_1\subseteq\Sc{V}_2$ and $\Sc{V}_2\subseteq\Sc{V}_1$. Hence $\Sc{V}_1=\Sc{V}_2$.
  \end{proof}

  \begin{definition}\label{Phi map}
	Let $E$ be a $C^*$-correspondence over a $C^*$-algebra $\Sc{A}$. When $(A,\sigma)$ is a representation of $E$ on $\Sc{H}$,
  we denote by $\Phi_A$ the completely positive map from $\sigma(\Sc{A})'$ to $\sigma(\Sc{A})'$ defined by
  \begin{equation*}
   \Phi_A(X)=\tilde{A}(I\otimes X)\tilde{A}^*
  \end{equation*}
  for every $X$ in $\sigma_A(\Sc{A})'$.
 \end{definition}

 \begin{remark}
  For any $a\in\Sc{A}$ and $X\in\sigma(\Sc{A})'$ we have
  \begin{align*}
   \sigma(a)\Phi_A(X)&=\sigma(a)\tilde{A}(I\otimes X)\tilde{A}^*=\tilde{A}\sigma^E(a)(I\otimes X)\tilde{A}^*\\
   &=\tilde{A}(I\otimes X)\sigma^E(a)\tilde{A}^*=\tilde{A}(I\otimes X)\tilde{A}^*\sigma(a).
  \end{align*}
  So $\Phi_A$ maps from $\sigma(\Sc{A})'$ to $\sigma(\Sc{A})'$ as claimed.
  \end{remark}

  In \cite{MuhlySolel2}  isometric representations that are not necessarily fully coisometric are studied. It is shown there that the corresponding $\Phi_A$ function for an isometric representation $(A,\sigma)$ will be an endomorphism of $\sigma(\Sc{A})'$. It is also shown that the fixed point set of $\Phi_A$ is the commutant of $\F{A}$, where $\F{A}$ is the algebra generated by the representation. In our setting, when $(A,\sigma)$ is a finite dimensional, fully coisometric representation on a Hilbert space $\Sc{V}$, we get that the commutant of $\F{A}$ is fixed by $\Phi_A$ (Lemma \ref{commutantfixedpoint}). Later, in Lemma \ref{commutantfixedpoint2}, when we compress $(A,\sigma)$ to a certain $\F{A}^*$-invariant subspace $\hat{\Sc{V}}\subseteq\Sc{V}$, we will get that the fixed point set of the corresponding $\Phi_{\hat{A}}$ map for the compressed representation $(\hat{A},\hat{\sigma}):=(P_{\hat{\Sc{V}}}A(\cdot)|_{\hat{\Sc{V}}},\sigma(\cdot)|_{\hat{\Sc{V}}})$, is the commutant of the compression of $\F{A}$ to $\hat{\Sc{V}}$.

  The map $\Phi_A$ is a generalisation of the map $\Phi$ introduced in Section $4$ of \cite{DavKriShp}. Indeed Lemma \ref{davkrishpLemma5.10} and Lemma \ref{davkrishpLemma5.11} are direct analogues of \cite[Lemma 5.10]{DavKriShp} and \cite[Lemma 5.11]{DavKriShp} respectively. We follow the same line of proof as in \cite{DavKriShp} when proving these results.
 
 \begin{lemma}\label{commutantfixedpoint}
	Let $(A,\sigma)$ be a fully coisometric representation of a $C^*$-correspond\-ence $E$ over a $C^*$-algebra $\Sc{A}$ on a finite dimensional Hilbert space $\Sc{V}$. Let $\F{A}$ be the unital algebra generated by the representation $(A,\sigma)$ and let $\Phi_A$ be the map from $\sigma(\Sc{A})'$ to  $\sigma(\Sc{A})'$ defined in Definition \ref{Phi map}.

  Then if $X$ is in the commutant of $\F{A}$, $X$ is a fixed point of $\Phi_A$.
 \end{lemma}

 \begin{proof}
  Suppose that $X\in\F{A}'$. Then for any $\xi\in E$ and $v\in V$ we have
  \begin{align*}
   X\tilde{A}(\xi\otimes v)&=XA(\xi)v=A(\xi)Xv\\
  &=\tilde{A}(\xi\otimes Xv)=\tilde{A}(I\otimes X)(\xi\otimes v).   
  \end{align*}
  Hence $X\tilde{A}=\tilde{A}(I\otimes X)$. Multiplying on the right by $\tilde{A}^*$ gives $X=\Phi_A(X)$.
 \end{proof}

 \begin{lemma}\label{davkrishpLemma5.10}
	Let $(A,\sigma)$ be a fully coisometric representation of a $C^*$-correspond\-ence $E$ over a $C^*$-algebra $\Sc{A}$ on a finite dimensional Hilbert space $\Sc{V}$. Let $\F{A}$ be the unital algebra generated by the representation $(A,\sigma)$ and let $\Phi_A$ be the map from $\sigma(\Sc{A})'$ to  $\sigma(\Sc{A})'$ defined in Definition \ref{Phi map}.

  Suppose there is an $X\in\sigma(\Sc{A})'$ which is non-scalar and $\Phi_A(X)=X$. Then $\Sc{V}$ has two pairwise orthogonal minimal $\F{A}^*$-invariant subspaces.
 \end{lemma}

 \begin{proof}
  Since $\Phi_A$ is unital and self-adjoint there is a positive, non-scalar $X\in\sigma(\Sc{A})'$ such that $\Phi_A(X)=X$. Assume $\|X\|=1$. Note that, as $X\in\sigma(\Sc{A})'$, the eigenspaces of $X$ are invariant under $\sigma(\Sc{A})$. Let $\mu$ be the smallest eigenvalue of $X$ and let $\Sc{M}=\ker(X-I)$ and $\Sc{N}=\ker(X-\mu I)$. Take any non-zero $x\in\Sc{M}$.
  \begin{align*}
   \|x\|^2&=\<\Phi_A(X)x,x\>=\<(I\otimes X)\tilde{A}^*x,\tilde{A}^*x\>\\
   &\leq\<\tilde{A}^*x,\tilde{A}^*x\>=\|x\|^2.
  \end{align*}
  From this we must have $(I\otimes X)\tilde{A}^*x=\tilde{A}^*x$ and hence $\tilde{A}^*x\in E\otimes\Sc{M}$.

  Note that if $x,y$ are eigenvectors for $X$ for different eigenvalues then
  \begin{equation*}
   \<\xi\otimes x,\eta\otimes y\>=\<x,\sigma(\<\xi,\eta\>)y\>=0,
  \end{equation*}
  for any $\xi,\eta\in E$. Hence if we take any non-zero $x\in\Sc{M}$ and let $y$ be any eigenvector of $X$ orthogonal to $\Sc{M}$ we get
  \begin{equation*}
   \<A(\xi)^*x,y\>=\<\tilde{A}^*x,\xi\otimes y\>=0
  \end{equation*}
  for any $\xi\in E$. Hence $\Sc{M}$ is $\F{A}^*$-invariant. The same argument works for $\Sc{N}$, as both $\Sc{M}$ and $\Sc{N}$ are eigenspaces for extremal values in the spectrum of $X$. As $\Sc{M}$ and $\Sc{N}$ are distinct eigenspaces for a self-adjoint operator, they are orthogonal. Since $\Sc{V}$ is a finite dimensional, there exists a space $\{0\}\neq\Sc{M}'\subseteq\Sc{M}$ of minimal dimension which is $\F{A}^*$-invariant and a space $\{0\}\neq\Sc{N}'\subseteq\Sc{N}$ of minimal dimension which is $\F{A}^*$-invariant.
 \end{proof}

 \begin{lemma}\label{davkrishpLemma5.11}
	Let $(A,\sigma)$ be a fully coisometric representation of a $C^*$-correspond\-ence $E$ over a $C^*$-algebra $\Sc{A}$ on a finite dimensional Hilbert space $\Sc{V}$. Let $\F{A}$ be the unital algebra generated by the representation $(A,\sigma)$ and let $\Phi_A$ be the map from $\sigma(\Sc{A})'$ to  $\sigma(\Sc{A})'$ defined in Definition \ref{Phi map}.

  Suppose $\Sc{V}=\Sc{V}_1\oplus\Sc{V}_2$ where both $\Sc{V}_1$ and $\Sc{V}_2$ are minimal $\F{A}^*$-invariant subspaces. Further suppose the representation $(A,\sigma)$ decomposes into $(B,\sigma_1)\oplus (C,\sigma_2)$ with respect to $\Sc{V}_1\oplus\Sc{V}_2$ with $\Sc{B}(\Sc{V}_1)=\Alg\{B(\xi),\sigma_1(a):\xi\in E, a\in\Sc{A}\}$ and $\Sc{B}(\Sc{V}_2)=\Alg\{C(\xi),\sigma_2(a):\xi\in E, a\in\Sc{A}\}$.

  If there exists $X\in\sigma(\Sc{A})'$ such that
  \begin{enumerate}
   \item $\Phi_A(X)=X$ and
   \item $X_{21}:=P_{\Sc{V}_2}XP_{\Sc{V}_1}\neq0$
  \end{enumerate}
  then there is a unitary $W$ such that
  \begin{equation*}
   C(\xi)=W^*B(\xi)W
  \end{equation*}
  and
  \begin{equation*}
   \sigma_2(a)=W^*\sigma_1(a)W
  \end{equation*}
  for all $\xi\in E$ and $a\in\Sc{A}$. Moreover the fixed point set of $\Phi_A$ consists of all matrices of the form $\bigr[\begin{smallmatrix}a_{11}I_{\Sc{V}_1}&a_{12}W^*\\a_{21}W& a_{22}I_{\Sc{V}_2}\end{smallmatrix}\bigl]$.
 \end{lemma}

 \begin{proof}
  We can assume that $X=X^*$ and $\|X_{21}\|=1$ as $\Phi_A$ is self-adjoint. We denote by $\tilde{B}$ and $\tilde{C}$ the usual maps from $E\otimes\Sc{V}_1$ and $E\otimes\Sc{V}_2$ respectively. Let $\Sc{M}=\{x\in\Sc{V}_1:\|X_{21}v\|=\|v\|\}$. As $\Sc{V}$ is finite dimensional, $\Sc{M}$ is non-empty. Note that for any $v\in\Sc{M}$ we have $X_{21}^*X_{21}v=v$. It follows that $\Sc{M}$ is a subspace of $\Sc{V}_1$. Thus if $v\in\Sc{M}$ and $a\in\Sc{A}$ we have
  \begin{equation*}
   \|X_{21}\sigma(a)v\|^2=\<X_{21}\sigma(a)v,X_{21}\sigma(a)v\>=\<X_{21}^*X_{21}v,\sigma(a^*a)v\>=\|\sigma(a)v\|^2.
  \end{equation*}
  So $\Sc{M}$ reduces $\sigma(\Sc{A})$. This tells us that $E\otimes\Sc{M}$ and $E\otimes(\Sc{V}_1\ominus\Sc{M})$ are orthogonal spaces.

  Now take any $v$ in $\Sc{M}$. We have that
  \begin{equation*}
   X_{21}v=\tilde{C}(I\otimes X_{21})\tilde{B}^*v.
  \end{equation*}
  This implies that $\|(I\otimes X_{21})\tilde{B}^*v\|=\|\tilde{B}^*v\|=\|v\|$. Thus $\tilde{B}^*v\in E\otimes\Sc{M}$ for all $v\in\Sc{M}$. Take any $\xi\in E$, $v\in\Sc{M}$ and $w\in\Sc{V}_1\ominus\Sc{M}$.
  \begin{equation*}
   \<B(\xi)^*v,w\>=\<\tilde{B}^*v,\xi\otimes w\>=0.
  \end{equation*}
  Thus $B(\xi)^*v\in\Sc{M}$. We conclude that $\Sc{M}$ is $\F{A}^*$-invariant. Hence, by the minimality of $\Sc{V}_1$, $\Sc{M}$ is all of $\Sc{V}_1$. Therefore $X_{21}$ is a unitary. Let $W=X_{21}$. For $v\in\Sc{V}_1$
  \begin{equation*}
   \|v\|=\|Wv\|=\|\tilde{C}(I\otimes W)\tilde{B}^*v\|\leq\|(I\otimes W)\tilde{B}^*v\|\leq\|v\|.
  \end{equation*}
  Hence $\tilde{C}$ is an isometry from Ran$(I\otimes W)\tilde{B}^*$ to the Ran$W=\Sc{V}_2$. $\tilde{C}$ is a contraction and so must be zero on the orthogonal complement of Ran$(I\otimes W)\tilde{B}^*$. It follows that $\tilde{C}^*$ is an isometry from $\Sc{V}_2$ to Ran$(I\otimes W)\tilde{B}^*$. Hence $\tilde{C}^*W=(I\otimes W)\tilde{B}^*$. From this it follows that $C(\xi)^*=WB(\xi)^*W^*$ for all $\xi\in E$. Since $W$ is also in the commutant of $\sigma(\Sc{A})$ it is the desired unitary.

  Suppose $Y\in\Sc{B}(\Sc{V}_1,\Sc{V}_2)$ and $\bigr[\begin{smallmatrix}0&0\\Y&0\end{smallmatrix}\bigl]$ is fixed by $\Phi_A$, then
  \begin{equation*}
   Y=\tilde{C}(I\otimes Y)\tilde{B}^*=W\tilde{B}(I\otimes W^*)(I\otimes Y)\tilde{B}^*=W\tilde{B}(I\otimes W^*Y)\tilde{B}^*.
  \end{equation*}
  It follows from Lemma \ref{davkrishpLemma5.10} that $W^*Y$ is a scalar and so $Y$ is a scalar multiple of $W$. A similar argument works for the other coordinates.
 \end{proof}  

  By Proposition \ref{reducingspacecor}, if $\Sc{V}'$ is an $\F{A}^*$-invariant subspace of $\Sc{V}$ such that $\F{A}[\Sc{V}']=\Sc{V}$ (i.e. $\Sc{V}'$ is cyclic for $\F{A}$) then $\F{S}[\Sc{V}']=\Sc{H}$. Hence the minimal isometric dilation of the completely contractive representation $(P_{\Sc{V}'}A(\cdot)|_{\Sc{V}'},\sigma(\cdot)|_{\Sc{V}'})$ is $(S,\rho)$.

  \begin{definition}
	Suppose $\F{A}$ is an algebra acting on a Hilbert space $\Sc{V}$, and that $\Sc{V}'$ is an $\F{A}^*$-invariant subspace of $\Sc{V}$ which is cyclic for $\F{A}$. If $\Sc{V}'$ has no proper $\F{A}^*$-invariant subspaces which are cyclic for $\F{A}$ then we say that $\Sc{V}'$ is a \emph{minimal cyclic coinvariant subspace (for $\F{A}$) of $\Sc{V}$.}

	When $(A,\sigma)$ is representation of a $C^*$-correspondence on a Hilbert space $\Sc{V}$ and $\F{A}$ is the unital \textsc{wot}-closed algebra generated by $(A,\sigma)$, we call a minimal cyclic coinvariant subspace for $\F{A}$ a \emph{minimal cyclic coinvariant subspace for $(A,\sigma)$.}
  \end{definition}

  The following proof is due to Ken Davidson.

  \begin{lemma}\label{cstaralg}
   Let $\Sc{V}$ be a finite dimensional Hilbert space. Suppose $\F{A}\subseteq\Sc{B}(\Sc{V})$ is an algebra and that $\Sc{V}$ is a minimal cyclic coinvariant space for $\F{A}$. Then $\F{A}$ is a $C^*$-algebra.
  \end{lemma}

  \begin{proof}
   Suppose $\Sc{L}$ is an $\F{A}^*$-invariant subspace such that $\Sc{V}\ominus\Sc{L}$ is not $\F{A}^*$-invariant. Let $\Sc{M}=\F{A}[\Sc{L}]$. Then $\Sc{L}\subsetneq\Sc{M}$ and $\Sc{M}\subsetneq\Sc{V}$. So $\Sc{V}\ominus\Sc{M}$ is a non-zero $\F{A}^*$-invariant subspace such that $\Sc{V}\ominus\Sc{M}\subsetneq\Sc{V}\ominus\Sc{L}$. We have that $\Sc{L}\oplus(\Sc{V}\ominus\Sc{M})$ is an $\F{A}^*$-invariant subspace and $\F{A}[\Sc{L}\oplus(\Sc{V}\ominus\Sc{M})]=\Sc{V}$. Hence, by our assumption that $\Sc{V}$ is a minimal cyclic coinvariant space, $\Sc{V}=\Sc{L}\oplus(\Sc{V}\ominus\Sc{M})$. This is a contradiction. Hence if $\Sc{L}$ is an $\F{A}^*$-invariant subspace then $\Sc{V}\ominus\Sc{L}$ must also be $\F{A}^*$-invariant. Since $\Sc{V}$ is finite dimensional, it follows that $\F{A}$ is a $C^*$-algebra.
  \end{proof}

  \begin{lemma}\label{uniqueminimal2}
	Let $(A,\sigma)$ be a fully coisometric representation of a $C^*$-corr\-esp\-ondence $E$ on a finite dimensional Hilbert space $\Sc{V}$, and let $(S,\rho)$ be the unique minimal isometric dilation of $(A,\sigma)$ on a Hilbert space $\Sc{H}$. Let $\F{A}$ be the unital algebra generated by the representation $(A,\sigma)$ and let $\F{S}$ be the unital \textsc{wot}-closed algebra generated by $(S,\rho)$.

   If $\Sc{V}_1,\Sc{V}_2,\ldots,\Sc{V}_k$ is a maximal set of pairwise orthogonal minimal $\F{A}^*$-invar\-iant spaces of $\Sc{V}$ then $\hat{\Sc{V}}=\Sc{V}_1\oplus\ldots\oplus\Sc{V}_k$ is the unique minimal cyclic coinvariant subspace of $\Sc{V}$.
  \end{lemma}

  \begin{proof}
   Firstly, $\F{S}[\hat{\Sc{V}}]$ is a $\F{S}$-reducing subspace by Lemma \ref{reducingspace}. If $\F{S}[\hat{\Sc{V}}]$ is not all of $\Sc{H}$ then its orthogonal complement in $\Sc{H}$, $\Sc{M}$, is also a $\F{S}$-reducing space. By Lemma \ref{lemma41} $\Sc{M}\cap\Sc{V}$ is a non-zero $\F{A}^*$-invariant space orthogonal to each $\Sc{V}_j$ for $1\leq j\leq k$. This contradicts the maximality of our choice of $\Sc{V}_1,\ldots,\Sc{V}_k$. Hence, by Proposition \ref{reducingspacecor}, $\hat{\Sc{V}}$ is $\F{A}$-cyclic. Since each $\Sc{V}_j$ is a minimal $\F{A}^*$-invariant space and since the $\F{S}$-reducing spaces $\F{S}[\Sc{V}_j]$ are orthogonal by Lemma \ref{reducingspace} it follows that $\hat{\Sc{V}}$ is indeed a minimal cyclic coinvariant subspace of $\Sc{V}$.

   Now suppose that $\Sc{W}$ is an $\F{A}^*$-invariant subspace of $\Sc{V}$ such that $\F{S}[\Sc{W}]=\Sc{H}$, i.e. $\F{A}[\Sc{W}]=\Sc{V}$. Let $\Sc{H}_j=\F{S}[\Sc{V}_j]$ for each $j$. We have that $\Sc{H}_j\subseteq\F{S}[\Sc{W}]$ for each $j$ and hence $\Sc{H}_j\cap\Sc{W}$ is non-zero. But each $\Sc{H}_j$ is irreducible by Corollary \ref{lemma41corollary2} and hence $\Sc{V}_j$ is the unique minimal $\F{S}^*$-invariant subspace of $\Sc{H}_j$ by Corollary \ref{uniqueminimal}. It follows that $\Sc{V}_j$ is contained in $\Sc{H}_j\cap\Sc{W}$ for each $j$. Therefore $\hat{\Sc{V}}\subseteq\Sc{W}$.
  \end{proof}

  \begin{remark}\label{cstaralg remark}
	In \cite{DavKriShp}, Lemma \ref{cstaralg} is proved for the case when $\F{A}$ is the unital algebra generated by a finite dimensional, fully coisometric representation $(A,\sigma)$ of the $C^*$-correspondence $\B{C}^n$ over $\B{C}$ (\cite[Part of Theorem 5.13]{DavKriShp}). The proof uses  analysis of $\Phi_A$ and the fact that $\hat{\Sc{V}}$ is a direct sum of minimal $\F{A}^*$-invariant subspaces. We note that the proof presented here shows that the result is in fact just a general result about cyclic, coinvariant subspaces in finite-dimensions, independent of any deeper analysis.

   However, that the minimal cyclic coinvariant space is unique is not a general result in finite dimensional linear algebra. For example, the algebra
   \begin{equation*}
    \F{C}=\left\{\begin{bmatrix}\lambda&0\\ \gamma-\lambda&\gamma\end{bmatrix}:\lambda,\gamma\in\B{C}\right\}
   \end{equation*}
   in $\Sc{B}(\B{C}^2)$ has both $\{(x,0):x\in\B{C}\}$ and $\{(x,x):x\in\B{C}\}$ as minimal cyclic coinvariant spaces.

   While it is shown that the minimal cyclic coinvariant space $\hat{\Sc{V}}$ in Lemma \ref{uniqueminimal2} is unique, the decomposition of $\hat{\Sc{V}}$ into a direct sum of minimal coinvariant subspaces is not necessarily unique. For example, suppose $\F{A}^*$ has two 1-dimensional invariant, orthogonal subspaces $\Sc{V}_1$ and $\Sc{V}_2$ and that the representation $(P_{\Sc{V}_1}A(\cdot)|_\Sc{V},\sigma(\cdot)|_{\Sc{V}_1})$ is unitarily equivalent to $(P_{\Sc{V}_2}A(\cdot)|_\Sc{V},\sigma(\cdot)|_{\Sc{V}_2})$. Let $U$ be the unitary defining the equivalence. Take a unit vector $v_1\in\Sc{V}_1$ and let $v_2=Uv_1$. Then $\Sc{V}'_1=\spa\{v_1+v_2\}$ and $\Sc{V}'_2=\spa\{v_1-v_2\}$ are orthogonal, $\F{A}^*$-invariant subspaces and
  \begin{equation*}
    \Sc{V}_1\oplus\Sc{V}_2=\Sc{V}'_1\oplus\Sc{V}'_2.
  \end{equation*}
  \end{remark}

  We follow the argument given in \cite[Theorem 5.13]{DavKriShp} for the following result. This serves as a converse to Lemma \ref{commutantfixedpoint}.

  \begin{lemma}\label{commutantfixedpoint2}
	Let $(A,\sigma)$ be a fully coisometric representation of a $C^*$-correspond\-ence $E$ over a $C^*$-algebra $\Sc{A}$. Let $\F{A}$ be the unital algebra generated by the representation $(A,\sigma)$ and let $\Phi_A$ be the map from $\sigma(\Sc{A})'$ to  $\sigma(\Sc{A})'$ defined in Definition \ref{Phi map}.

	Suppose $\Sc{V}=\hat{\Sc{V}}$, where $\hat{\Sc{V}}=\Sc{V}_1\oplus\ldots\oplus\Sc{V}_k$ is as in Lemma $\ref{uniqueminimal2}$. Then the fixed point set of $\Phi_A$ is equal to the commutant of $\F{A}$.
  \end{lemma}

  \begin{proof}
   We have already shown in Lemma \ref{commutantfixedpoint} that if $X\in\F{A}'$ then $\Phi_A(X)=X$. Take $X\in\sigma(\Sc{A})'$ such that $\Phi_A(X)=X$. Suppose that $X$ is non-scalar. If there is no unitary between $\Sc{V}_k$ and $\Sc{V}_l$ intertwining $\F{A}$ then, by Lemma \ref{davkrishpLemma5.11}, $P_{\Sc{V}_k}XP_{\Sc{V}_l}=0$. On the other hand, if $W_{k,l}$ is an intertwining unitary from $\Sc{V}_k$ to $\Sc{V}_l$ then Lemma \ref{davkrishpLemma5.11} tells us that $P_{\Sc{V}_k}XP_{\Sc{V}_l}=x_{kl}W_{k,l}$ for some scalar $x_{kl}$, and hence $P_{\Sc{V}_k}XP_{\Sc{V}_l}$ is in $\F{A}'$. It follows that $X\in\F{A}'$.
  \end{proof}

  The following theorem summarises our main results.

  \begin{theorem}\label{summary}
   Suppose $E$ is a $C^*$-correspondence over a $C^*$-algebra $\Sc{A}$. Let $(A,\sigma)$ be a fully coisometric, finite dimensional representation of $E$ on a Hilbert space $\Sc{V}$, and let $(S,\rho)$ be the minimal isometric dilation of $(A,\sigma)$ on $\Sc{H}$. Let $\F{A}$ be the unital algebra generated by $(A,\sigma)$ and $\F{S}$ be the unital \textsc{wot}-closed algebra generated by $(S,\rho)$.

   If
   \begin{equation*}
    \hat{\Sc{V}}=\sideset{}{^\oplus}\sum_{j=1}^n\Sc{V}_j
   \end{equation*}
   is a maximal direct sum of minimal, orthogonal $\F{A}^*$-invariant subspaces of $\Sc{V}$, then $\hat{\Sc{V}}$ is the \emph{unique} minimal $\F{A}^*$-invariant subspace such that $\F{S}[\hat{\Sc{V}}]=\Sc{H}$. Further 
   \begin{equation*}
    \Sc{H}=\sideset{}{^\oplus}\sum_{j=1}^n\Sc{H}_j
   \end{equation*}
   where $\Sc{H}_j=\F{S}[\Sc{V}_j]$.

   The representation $(P_{\hat{\Sc{V}}^\perp}S(\cdot)|_{\hat{\Sc{V}}^\perp},\rho(\cdot)|_{\hat{\Sc{V}}^\perp})$ is an induced representation and $\F{S}^*|_{\hat{\Sc{V}}}$ is a $C^*$-algebra.
  \end{theorem}
  
  We now show that the compression to the minimal cyclic coinvariant space for a finitely correlated, fully coisometric representation is a complete unitary invariant.

  \begin{corollary}\label{complete unitary invariant}
   Suppose $(S,\sigma)$ and $(T,\tau)$ are finitely correlated, isometric, fully coisometric representations of a $C^*$-correspondence $E$ on $\Sc{H}_S$ and $\Sc{H}_T$ respectively. Let $\Sc{V}_S$ be the unique minimal cyclic coinvariant subspace for $(S,\sigma)$ and let $\Sc{V}_T$ be the unique minimal cyclic subspace for $(T,\tau)$.

   Then $(S,\sigma)$ and $(T,\tau)$ are unitarily equivalent if and only if the finite dimensional representations $(P_{\Sc{V}_S}S(\cdot)|_{\Sc{V}_S},\sigma(\cdot)|_{\Sc{V}_S})$ and $(P_{\Sc{V}_T}T(\cdot)|_{\Sc{V}_T},\tau(\cdot)|_{\Sc{V}_T})$ are unitarily equivalent.
  \end{corollary}

  \begin{proof}
   Suppose $(S,\sigma)$ and $(T,\tau)$ are unitarily equivalent. Let $U$ be the unitary from $\Sc{H}_S$ to $\Sc{H}_T$ intertwining $(S,\sigma)$ and $(T,\tau)$. It follows that $U\Sc{V}_S$ is invariant under $T(\cdot)^*$ and is cyclic, hence $\Sc{V}_T\subseteq U\Sc{V}_S$. Similarly $\Sc{V}_S\subseteq U^*\Sc{V}_T$. It follows that $U\Sc{V}_S=\Sc{V}_T$ and $(P_{\Sc{V}_S}S(\cdot)|_{\Sc{V}_S},\sigma(\cdot)|_{\Sc{V}_S})$ and $(P_{\Sc{V}_T}T(\cdot)|_{\Sc{V}_T},\tau(\cdot)|_{\Sc{V}_T})$ are unitarily equivalent.

   Conversely, suppose that $(P_{\Sc{V}_S}S(\cdot)|_{\Sc{V}_S},\sigma(\cdot)|_{\Sc{V}_S})$ and $(P_{\Sc{V}_T}T(\cdot)|_{\Sc{V}_T},\tau(\cdot)|_{\Sc{V}_T})$ are unitarily equivalent. Then, by the uniqueness of the minimal isometric dilation, $(S,\sigma)$ and $(T,\tau)$ are unitarily equivalent.
  \end{proof}

  \section[Product Systems of $C^*$-correspondences]{\texorpdfstring{Product Systems of $C^*$-correspondences over $\mathbb{N}^k$}{Product Systems of C*-correspondences}}\label{prod sys section}
		We will now extend our results to product systems of $C^*$-correspondences. This is the analogue of multivariate operator theory, and so relies on a more sophisticated dilation theory. The key to our anaylsis will be a trick to reduce to the consideration of a certain $C^*$-correspondence contained inside our product system (Theorem \ref{rank k to rank 1}).

  \subsection{Preliminaries and Notation}
	Recall that we are restricting our attention to unital $C^*$-algebras, and we are only considering unital representations of $C^*$-algebras.

  The following description of product systems of $C^*$-corres\-pondences over $\B{N}^k$ follows that of \cite{Fowler} and \cite{Solel3}. Let $\Sc{A}$ be a unital $C^*$-algebra. A semigroup $E$ is a \emph{product system of $C^*$-correspondences over $\B{N}^k$} if there is a semigroup homomorphism $p:E\rightarrow\B{N}^k$ such that $E(\mathbf{n}):=p^{-1}(\mathbf{n})$ is a $C^*$-correspondence over $\Sc{A}$ and the map $(\xi,\eta)\in E(\mathbf{n})\times E(\mathbf{m})\rightarrow\xi\eta\in E(\mathbf{n}+\mathbf{m})$ extends to an isomorphism $t_{\mathbf{n},\mathbf{m}}$ from $E(\mathbf{n})\otimes E(\mathbf{m})$ onto $E(\mathbf{n}+\mathbf{m})$. By $E(\mathbf{0})$ we mean the $C^*$-algebra $\Sc{A}$. Letting $\mathbf{e}_1, \mathbf{e}_2,\ldots,\mathbf{e}_k$ be the standard generating set of $\B{N}^k$, we write $E_i$ for the $C^*$-correspondence $p^{-1}(\mathbf{e}_i)$. We identify $E(\mathbf{n})$ with $E_1^{n_1}\otimes\ldots\otimes E_k^{n_k}$ when $\mathbf{n}=(n_1,\ldots,n_k)$. It follows that $t_{i,j}:=t_{\mathbf{e}_i,\mathbf{e}_j}$ is an isomorphism from $E_i\otimes E_j$ to $E_j\otimes E_i$, for $i\leq j$ and $t_{j,i}=t_{i,j}^{-1}$ for $i\leq j$. We write $t_{i,i}$ for the identity on $E_i^2$. We will often suppress the isomorphism and write $E(\mathbf{n})\otimes E(\mathbf{m})=E(\mathbf{n}+\mathbf{m})$.

  If, for each $i$, $(A^{(i)},\sigma)$ is a representation of $E_i$ on a Hilbert space $\Sc{H}$ and we have the following commutation relation
  \begin{equation*}
   \tilde{A}^{(i)}(I_{E_i}\otimes\tilde{A}^{(j)})=\tilde{A}^{(j)}(I_{E_j}\otimes\tilde{A}^{(i)})(t_{i,j}\otimes I_{\Sc{H}})
  \end{equation*}
  then $(A^{(1)},\ldots,A^{(k)},\sigma)$ is a (\emph{completely contractive covariant$)$ representation of $E$ on $\Sc{H}$}. A representation $(A^{(1)},\ldots,A^{(k)},\sigma)$ is said to be \emph{isometric} (resp. \emph{fully coisometric}) if each representation $(A^{(i)},\sigma)$ is isometric (resp. fully coisometric).

  For $\mathbf{n}=(n_1,\ldots,n_k)\in\B{N}^k$ we define a map $\tilde{A}_\mathbf{n}$ from $E(\mathbf{n})\otimes\Sc{H}$ to $\Sc{H}$ by
  \begin{align*}
   \tilde{A}_\mathbf{n}&=\tilde{A}^{(1)}_{n_1}(I_{E^{n_1}_1}\otimes\tilde{A}^{(2)}_{n_2})\ldots(I_{E^{n_1}_1}\otimes\ldots\otimes I_{E^{n_{k-1}}_{k-1}}\otimes\tilde{A}^{(k)}_{n_k})
  \end{align*}
  We define a representation $(A_\mathbf{n},\sigma)$ of the $C^*$-correspondence $E(\mathbf{n})$ by letting
  \begin{equation*}
   A_\mathbf{n}(\xi)h=\tilde{A}_\mathbf{n}(\xi\otimes h)
  \end{equation*}
  for each $\xi\in E(\mathbf{n})$ and $h\in\Sc{H}$.

  A representation $(A^{(1)},\ldots,A^{(k)},\sigma)$ of $E$ is said to be \emph{doubly commuting} if it satisfies
  \begin{equation*}
   \tilde{A}^{(j)*}\tilde{A}^{(i)}=(I_{E_j}\otimes\tilde{A}^{(i)})(t_{i,j}\otimes I_{\Sc{H}})(I_{E_i}\otimes\tilde{A}^{(j)*}).
  \end{equation*}
  It has been shown in \cite{Fowler} and \cite{Solel3} that the doubly commuting condition is equivalent to what is known as \emph{Nica covariance} \cite{Nica}. It is easy to check that an isometric, fully coisometric representation is doubly commuting.

  Note that if $(A^{(1)},\ldots,A^{(k)},\sigma)$ is an isometric representation, then for $\mathbf{n}=(n_1,\ldots,n_k)$, $\mathbf{m}=(m_1,\ldots,m_k)\in\B{N}^k$ we have
  \begin{align*}
   \tilde{A}_\mathbf{m}^*\tilde{A}_\mathbf{n}&=I_{E(\mathbf{n}-(\mathbf{n}-\mathbf{m})_+)}\otimes\tilde{A}_{(\mathbf{n}-\mathbf{m})_{-}}^*\tilde{A}_{(\mathbf{n}-\mathbf{m})_+}
  \end{align*}
  where $(\mathbf{n}-\mathbf{m})_+$ is equal to $n_i-m_i$ in the $i^{th}$ coordinate if $n_i\geq m_i$ and zero in the $i^{th}$ coordinate otherwise, and $(\mathbf{n}-\mathbf{m})_{-}\in\B{N}^k$ satisfies $\mathbf{n}-\mathbf{m}=(\mathbf{n}-\mathbf{m})_+-(\mathbf{n}-\mathbf{m})_-$.

  We define the Fock space $\Sc{F}(E)$ of a product space of $C^*$-correspondences by
  \begin{equation*}
   \Sc{F}(E)=\sideset{}{^\oplus}\sum_{\mathbf{n}\in\B{N}^k}E(\mathbf{n}).
  \end{equation*}
  For more details on the construction see \cite{Fowler}. For each $\mathbf{n}$ and $\xi\in E(\mathbf{n})$ define the \emph{creation operator} $T_\xi:\Sc{F}(E)\rightarrow\Sc{F}(E)$ by
  \begin{equation*}
   T_\xi(\eta)=\xi\otimes\eta
  \end{equation*}
  for each $\eta\in\Sc{F}(E)$. The $C^*$-algebra in $\Sc{L}(\Sc{F}(E))$ generated by the creation operators is called the \emph{Toeplitz algebra} associated to $E$ and denoted $\Sc{T}(E)$. A product system $(E,\Sc{A})$ is said to have the \emph{normal ordering property} if
  \begin{equation*}
   \Sc{T}(E)=\overline{\spa}\{L(\xi)L(\eta)^*:\xi,\eta\in\cup_{\mathbf{n}\in\B{N}^k}E(\mathbf{n})\}.
  \end{equation*}

	Let $(S^{(1)},\ldots,S^{(k)},\rho)$ be a representation of a product system $(E,\Sc{A})$ on a Hilbert space $\Sc{H}$. We denote by $I$ be the identity in $\Sc{B}(H)$. We call the weak-operator topology closed algebra
	\begin{equation*}
		\F{S}=\Alg\{I,\ S^{(i)}(\xi_i),\ \rho(a):a\in\Sc{A},\ \xi_i\in E_i\text{ for } 1\leq i\overline{\leq k\}}^{\textsc{wot}}
	\end{equation*}
the \emph{unital \textsc{wot}-closed algebra generated by the representation $(S^{(1)},\ldots,S^{(k)},\rho)$.}

  \subsection{Minimal Isometric Dilations}\label{prod sys min iso dil section}
  \begin{definition}
  Let $(E,\Sc{A})$ be a product system over $\B{N}^k$ and let $(A^{(1)},\ldots,A^{(k)},\sigma)$ be a representation of $E$ on $\Sc{V}$. A representation $(S^{(1)},\ldots,S^{(k)},\rho)$ on a Hilbert space $\Sc{H}$ is a \emph{dilation} of $(A^{(1)},\ldots,A^{(k)},\sigma)$ if $\Sc{H}$ contains $\Sc{V}$ and, for each $i$, $(S^{(i)},\rho)$ dilates $(A^{(i)},\sigma)$. A dilation $(S^{(1)},\ldots,S^{(k)},\rho)$ of $(A^{(1)},\ldots,A^{(k)},\sigma)$ is an \emph{isometric} dilation if $(S^{(1)},\ldots,S^{(k)},\rho)$ is an isometric representation. A dilation $(S^{(1)},\ldots,S^{(k)},\rho)$ of $(A^{(1)},\ldots,A^{(k)},\sigma)$ on $\Sc{H}$ is \emph{minimal} if
  \begin{equation*}
   \Sc{H}=\bigvee_{\mathbf{n}\in\B{N}^k}\tilde{S}_\mathbf{n}(E(\mathbf{n})\otimes\Sc{V}).
  \end{equation*}
  \end{definition}

  Given an arbitrary representation of a product system $(E,\Sc{A})$ over $\B{N}^k$ it is not always possible to find an isometric dilation. Indeed, if $k\geq 3$ and $\Sc{A}=E=\B{C}$, then a representation of $E$ is simply $k$ commuting contractions $A_1,\ldots,A_k$. It is known that there are examples of commuting contractions which can not be dilated to commuting isometries, see e.g. \cite{Parrott}. With that said, there are a number of dilation theorems for product systems of $C^*$-correspondences. We will now review a number of these dilations results that will be useful. The susbsequent remarks may help clarify some of the distinctions.
  
  \begin{theorem}[Solel \cite{Solel2}]\label{N^2}
   Let $(E,\Sc{A})$ be a product system of $C^*$-correspondences over $\B{N}^2$. Then any representation of $E$ has an isometric dilation.
  \end{theorem}

  \begin{definition}
   Let $(A^{(1)},\ldots,A^{(k)},\sigma)$ be a representation of a product system $(E,\Sc{A})$ on $\Sc{H}$. For each $\mathbf{n}\in\B{Z}^k$ we define $A(\mathbf{n})$ to be
	\begin{equation*}
		A(\mathbf{n})=\tilde{A}_{\mathbf{n}_-}^*\tilde{A}_{\mathbf{n}_+}.
	\end{equation*}
   Let $(S^{(1)},\ldots,S^{(k)},\rho)$ be an isometric dilation of $(A^{(1)},\ldots,A^{(k)},\sigma)$. If for each $\mathbf{n}\in\B{Z}^k$, $(S^{(1)},\ldots,S^{(k)},\rho)$ satisfies
	\begin{equation*}
		(I_{E(\mathbf{n}_+)}\otimes P_{\Sc{H}})S(\mathbf{n})|_{E(\mathbf{n}_+)\otimes\Sc{H}}=A(\mathbf{n})
	\end{equation*}
   then $(S^{(1)},\ldots,S^{(k)},\rho)$ is a \emph{regular isometric dilation} of $(A^{(1)},\ldots,A^{(k)},\sigma)$.
  \end{definition}

  \begin{theorem}[Solel \cite{Solel3}]\label{regular dilation}
   Let $(E,\Sc{A})$ be a product system of $C^*$-correspondences over $\B{N}^k$ and let $(A^{(1)},\ldots,A^{(k)},\sigma)$ be a representation of $E$. If $(A^{(1)},\ldots,A^{(k)},\sigma)$ satisfies the additional condition that, for every $v\subseteq\{1,\ldots,k\}$
   \begin{equation}\label{hasregulardilation}
    \sum_{u\subseteq v}(-1)^{|u|}(I_{\mathbf{e}(v)-\mathbf{e}(u)}\otimes\tilde{A}_{\mathbf{e}(u)}^*\tilde{A}_{\mathbf{e}(u)})\geq0,
   \end{equation}
   where $\mathbf{e}(u)\in\B{N}^k$ is $1$ in the $i^{th}$ coordinate if $i\in u$ and zero in the $i^{th}$ coordinate otherwise, then it has a unique minimal regular isometric dilation.
  \end{theorem}

  \begin{theorem}[Solel \cite{Solel3}]\label{regular dilation of doubly commuting}
   Let $(E,\Sc{A})$ be a product system of $C^*$-correspondences over $\B{N}^k$ and let $(A^{(1)},\ldots,A^{(k)},\sigma)$ be a doubly commuting representation of $E$. Then $(A^{(1)},\ldots,A^{(k)},\sigma)$ will satisfy \eqref{hasregulardilation}. Further, the minimal regular isometric dilation of $(A^{(1)},\ldots,A^{(k)},\sigma)$ will be doubly commuting.
  \end{theorem}

  \begin{theorem}[Shalit \cite{Shal2}]\label{coisometric dilations}
   Let $(E,\Sc{A})$ be a product system of $C^*$-correspondences over $\B{N}^k$ and let $(A^{(1)},\ldots,A^{(k)},\sigma)$ be a fully coisometric representation of $E$. Then $(A^{(1)},\ldots,A^{(k)},\sigma)$ has a minimal isometric dilation which is fully coisometric.
  \end{theorem}

  \begin{definition}
   Let $(E,\Sc{A})$ be product system of $C^*$-correspondences over $\B{N}^k$. For a representation $(A^{(1)},\ldots,A^{(k)},\sigma)$ of $E$ on a Hilbert space $\Sc{H}$ define the \emph{defect operator} for $s\in(0,1)$
   \begin{equation*}
    \Delta_s=\sum_{\substack{\mathbf{n}\in\B{N}^k\\\mathbf{n}\leq(1,1,\ldots,1)}}(-s^2)^{(|\mathbf{n}|)}\tilde{A}(\mathbf{n})\tilde{A}(\mathbf{n})^*,
   \end{equation*}
   where $|\mathbf{n}|=n_1+\ldots+n_k$ when $\mathbf{n}=(n_1,n_2,\ldots,n_k)$.

   The representation $(A^{(1)},\ldots,A^{(k)},\sigma)$ is said to satisfy \emph{the Popescu condition} if there is a $t\in(0,1)$ such that $\Delta_s$ is positive for all $s\in(t,1)$.
  \end{definition}

  \begin{theorem}[Skalski \cite{Skalski}]\label{SkalskiDil}
   Let $(E,\Sc{A})$ be a product system of $C^*$-correspondences over $\B{N}^k$ having the normal ordering property. Let $(A^{(1)},\ldots,A^{(k)},\sigma)$ be a representation of $E$. If $(A^{(1)},\ldots,A^{(k)},\sigma)$ satisfies the Popescu condition then it has an isometric dilation.
  \end{theorem}

  \begin{remark}[Remarks on Theorems \ref{N^2}, \ref{regular dilation}, \ref{coisometric dilations} and \ref{SkalskiDil}]
   The dilation given in Theorem \ref{N^2} is not necessarily unique. Examples of representations which do not dilate uniquely are given by Davidson, Power and Yang in \cite{DavPowYang1}. They also provide an alternative proof of Theorem \ref{N^2} for the case that $\Sc{A}=\B{C}$ and $E_i=\B{C}^{n_i}$ for $i=1,2$. Further it is proved that in this setting a minimal isometric dilation of a fully coisometric representation is fully coisometric and unique.

   A fully coisometric representation does not necessarily satisfy \eqref{hasregulardilation}. For example if $T_1=T_2=S^*$, where $S$ is a unilateral shift on a separable Hilbert space, then the commuting coisometries $T_1$ and $T_2$ do no satisfy \eqref{hasregulardilation}. The atomic representations of single vertex $k$-graphs studied in \cite{DavPowYang2, DavYang2} do satisfy \eqref{hasregulardilation} since they are doubly commuting. For another example of a non-doubly commuting, fully coisometric representation see Example \ref{not doubly commuting}.

   An alternative proof of Theorem \ref{regular dilation} was given by Shalit in \cite{Shal}. The method of proof in \cite{Shal} and \cite{Shal2} is to construct a semigroup of commuting contractions from a contractive representation. The result is then deduced from dilation results for semigroups of commuting contractions.

   Skalski and Zacharias \cite{SkalZach1} show that if $(A^{(1)},\ldots,A^{(k)},\sigma)$ is a doubly commuting representation of $E$ then its minimal isometric dilation is fully coisometric if and only if $(A^{(1)},\ldots,A^{(k)},\sigma)$ is fully coisometric. We will show in Lemma \ref{coiso2} that a minimal, isometric dilation of a representation $(A^{(1)},\ldots,A^{(k)},\sigma)$ is fully coisometric if and only if $(A^{(1)},\ldots,A^{(k)},\sigma)$ is fully coisometric, without the assumption that $(A^{(1)},\ldots,A^{(k)},\sigma)$ is doubly commuting.
  
   It is noted in \cite{Skalski} that if a representation $(A^{(1)},\ldots,A^{(k)},\sigma)$ is doubly commuting or coisometric then it will satisfy the Popescu condition. Theorem \ref{SkalskiDil} is a more general version of a dilation theorem for $k$-graphs proved by Skalski and Zacharias in \cite{SkalZach2}. We will look more closely at $k$-graphs in \S \ref{examples}.
  \end{remark}

  The following result is just a higher-rank version of Lemma \ref{coiso} and follows much the same argument.

  \begin{lemma}\label{coiso2}Let  $(A^{(1)},\ldots,A^{(k)},\sigma)$ be a representation of a product system $E$ on a Hilbert space $\Sc{V}$ with a minimal isometric dilation $(S^{(1)},\ldots,S^{(k)},\rho)$ on a Hilbert space $\Sc{H}$. Then $(S^{(1)},\ldots,S^{(k)},\rho)$ is fully coisometric if and only if $(A^{(1)},\ldots,A^{(k)},\sigma)$ is fully coisometric.
  \end{lemma}

  \begin{proof}
   That $(A^{(1)},\ldots,A^{(k)},\sigma)$ is fully coisometric when $(S^{(1)},\ldots,S^{(k)},\rho)$ is follows the same argument as in Lemma \ref{coiso}.

   Conversely, assume that $(A^{(1)},\ldots,A^{(k)},\sigma)$ is fully coisometric. We will show that $\tilde{S}:=\tilde{S}_1$ is a coisometry. That $\tilde{S}_i$ is a coisometry, for $2\leq i\leq k$, follows similarly. Note that $\tilde{S}$ is an isometry and so $\tilde{S}\tilde{S}^*$ is a projection on $\Sc{H}$. Let $\Sc{M}=(I-\tilde{S}\tilde{S}^*)\Sc{H}$. Take any $x\in\Sc{M}$ and $y\in\Sc{H}$. We have
  \begin{equation*}
   \<S(\xi_1)^*x,S(\xi_2)y\>=\<x,\tilde{S}(\xi_1\otimes S(\xi_2)y)\>=0
  \end{equation*}
  for all $\xi_1, \xi_2\in E_1$. For $2\leq i\leq k$ we have
  \begin{align*}
   \<S^{(i)}(\eta)^*x,S(\xi)y\>&=\<x,S^{(i)}(\eta)S(\xi)y\>\\
   &=\<x,\tilde{S}^{(i)}(I_{E_i}\otimes\tilde{S})(\eta\otimes\xi\otimes y)\>\\
   &=\<x,\tilde{S}(I_E\otimes\tilde{S}^{(i)})\circ(t\otimes I_\Sc{H})(\eta\otimes\xi\otimes y)\>\\
   &=0
  \end{align*}
  for all $\xi\in E$ and $\eta\in E_i$ (where $t=t_{1,i}$). It follows that $\Sc{M}$ is $\F{S}^*$-invariant, where $\F{S}$ is the unital \textsc{wot}-closed algebra generated by $(S^{(1)},\ldots,S^{(k)},\rho)$.
  The rest of the proof follows the same argument as Lemma \ref{coiso}.
  \end{proof}

  \begin{lemma}\label{fullycoiso uniq dil}
   Let $(A^{(1)},\ldots,A^{(k)},\sigma)$ be a fully coisometric representation of a product system $E$. Then the minimal isometric dilation of $(A^{(1)},\ldots,A^{(k)},\sigma)$ is unique up to unitary equivalence.
  \end{lemma}

  \begin{proof}
   Since $(A^{(1)},\ldots,A^{(k)},\sigma)$ is fully coisometric it can be dilated by Theorem \ref{coisometric dilations}. It follows from \cite[Theorem 2.7]{Skalski} that all doubly commuting, minimal, isometric dilations of a representation $(A^{(1)},\ldots,A^{(k)},\sigma)$ are unitarily equivalent. By Lemma \ref{coiso2}, if $(A^{(1)},\ldots,A^{(k)},\sigma)$ is a fully coisometric representation then all minimal, isometric dilations are also fully coisometric, and hence they are doubly commuting. It follows that the minimal isometric dilation is unique up to unitary equivalence.
  \end{proof}

	We now prove a key technical tool. We show that taking the minimal isometric dilation of a fully coisometric representation $(A^{(1)},\ldots,A^{(k)},\sigma)$ gives rise to the minimal isometric representation of the representation $(A_{\mathbf{n}},\sigma)$ when $\mathbf{n}\geq(1,\ldots,1)$. This allows us, in Lemma \ref{lemma41 rank k}, to prove the analogous result of Lemma \ref{lemma41} for product systems. In fact, Theorem \ref{rank k to rank 1} allows us to deduce Lemma \ref{lemma41 rank k} from Lemma \ref{lemma41}. Lemma \ref{lemma41 rank k} will play an important role in our analysis, just as Lemma \ref{lemma41} did in the study of the $C^*$-correspondence case.

  \begin{theorem}\label{rank k to rank 1}
   Let $(A^{(1)},\ldots,A^{(k)},\sigma)$ be a fully coisometric representation of a product system of $C^*$-correspondences $E$ on a Hilbert space $\Sc{V}$ with minimal isometric dilation $(S^{(1)},\ldots,S^{(k)},\rho)$ on a Hilbert space $\Sc{H}$. If $\mathbf{n}=(n_1,n_2,\ldots,n_k)\in\B{N}^k$ satisfies $n_i\neq0$ for $1\leq i\leq k$ then the $C^*$-correspondence representation $(S_\mathbf{n},\rho)$ of $E(\mathbf{n})$ is the $($unique$)$ minimal isometric dilation of $(A_\mathbf{n},\sigma)$.
  \end{theorem}

  \begin{proof}
   It is clear that $(S_\mathbf{n},\rho)$ is an isometric dilation of $(A_\mathbf{n},\sigma)$ for any $\mathbf{n}\in\B{N}^k$. It remains to show that the dilation is minimal when $n_i\neq0$ for each $i$.

   For any $\mathbf{n}\in\B{N}^k$ we define $\Sc{H}_\mathbf{n}$ to be the space mapped out by $(S_\mathbf{n},\sigma)$, i.e.
   \begin{equation*}
    \Sc{H}_\mathbf{n}=\bigvee_{\substack{m\in\B{Z}\\ m\geq0}}\tilde{S_\mathbf{n}}_m(E(\mathbf{n})^m\otimes\Sc{V}).
   \end{equation*}

   \begin{claim}[1]
    If $\mathbf{m},\mathbf{n}\in\B{N}^k$ and $\mathbf{m}\leq \mathbf{n}$, then $\Sc{H}_\mathbf{m}\subseteq\Sc{H}_\mathbf{n}$.
   \end{claim}
   Let $\mathbf{p}=\mathbf{n}-\mathbf{m}$. Take any $v\in\Sc{V}$ and $\xi\in E(\mathbf{m})$ then
   \begin{align*}
    \tilde{S}_\mathbf{m}(\xi\otimes v)&=\tilde{S}_\mathbf{m}(I_{E(\mathbf{m})}\otimes\tilde{S}_\mathbf{p})(I_{E(\mathbf{m})}\otimes\tilde{S}_\mathbf{p}^*)(\xi\otimes v)\\
    &\in\tilde{S}_\mathbf{n}(E(\mathbf{n})\otimes\Sc{V}).
   \end{align*}
   That the range of $\tilde{S}_\mathbf{m}^l$ is contained in the range of $\tilde{S}_\mathbf{n}^l$ for positive integers $l$ follows by a similar argument.

   \begin{claim}[2]
    If $\mathbf{m},\mathbf{n}\in\B{N}^k$ and $\mathbf{n}=l\mathbf{m}$ for some positive integer $l$, then $\Sc{H}_\mathbf{m}=\Sc{H}_\mathbf{n}$.
   \end{claim}
    We know from the first claim that $\Sc{H}_\mathbf{m}\subseteq\Sc{H}_\mathbf{n}$. The reverse inclusion follows from the fact that $\tilde{S}_\mathbf{n}$ is isomorphic to
    \begin{equation*}
     \tilde{S}_\mathbf{m}(I_{E(\mathbf{m})}\otimes\tilde{S}_\mathbf{m})\ldots(I_{E(\mathbf{m})^{p-1}}\otimes\tilde{S}_\mathbf{m}).
    \end{equation*}

   \begin{claim}[3]
    If $\mathbf{m},\mathbf{n}\in\B{N}^k$ such that $n_i,m_i\neq0$ for $1\leq i\leq k$, then $\Sc{H}_\mathbf{m}=\Sc{H}_\mathbf{n}$.
   \end{claim}
    Choose an integer $l$ such that $l\mathbf{m}\geq\mathbf{n}$. Then, by the previous two claims, $\Sc{H}_\mathbf{n}\subseteq\Sc{H}_{l\mathbf{m}}=\Sc{H}_\mathbf{m}$. The reverse inclusion follows similarly.

   Now, since $(S^{(1)},\ldots,S^{(k)},\rho)$ is a minimal dilation, we have that $\Sc{H}=\bigvee_{\mathbf{n}\in\B{N}^k}\Sc{H}_\mathbf{n}$. However, if we fix $\mathbf{n}\in\B{N}^k$ such that $n_i\neq 0$ for each $i$, then the previous three claims tell us that $\Sc{H}_\mathbf{m}\subseteq\Sc{H}_\mathbf{n}$ for every $\mathbf{m}\in\B{N}^k$. Hence $\Sc{H}=\Sc{H}_\mathbf{n}$ and so $(S_\mathbf{n},\rho)$ is the minimal isometric dilation of $(A_\mathbf{n},\sigma)$.
  \end{proof}

  \begin{remark}\label{rank k to rank 1 remark}
   The condition in Theorem \ref{rank k to rank 1} that $\mathbf{n}\geq (1,1,\ldots,1)$ is necessary to guarantee that $(S_\mathbf{n},\rho)$ is the minimal isometric dilation of $(A_\mathbf{n},\sigma)$. For example, let $\Sc{H}$ be a separable Hilbert space with orthonormal basis $\{e_n:n\geq0\}$. Define commuting isometries $T_1$ and $T_2$ on $\Sc{H}$ by
   \begin{align*}
    T_1e_n&=e_{2n} \ \text{ and}\\T_2e_n&=e_{3n}.
   \end{align*}
   Then $T_1^*$ and $T_2^*$ are commuting coisometries. Let $U_1$ and $U_2$ be the minimal commuting unitaries dilating $T_1^*$ and $T_2^*$. Note that commuting unitaries are necessarily doubly commuting. We have that for any $n,k\geq0$
   \begin{align*}
    \<U_1e_3,U_2^ke_n\>&=\<e_3,U_2^ke_{2n}\>=\<e_3,T_2^{*k}e_{2n}\>=0,
   \end{align*}
   and so $U_2$ is not the minimal isometric dilation of $T_2^*$.

   In the case of fully coisometric, atomic representations of single vertex $k$-graphs, however, it is not necessary for $\mathbf{n}\geq(1,\ldots,1)$ for Theorem \ref{rank k to rank 1} to be satisfied. See Example \ref{atomic flip}, or \cite{DavPowYang2, DavYang2}.
  \end{remark}

  We now prove a higher rank version of Lemma \ref{reducingspace}.
  \begin{lemma}\label{reducingspace2}
   Let  $(A^{(1)},\ldots,A^{(k)},\sigma)$ be a representation of a product system $E$ on a Hilbert space $\Sc{V}$ with a minimal isometric dilation $(S^{(1)},\ldots,S^{(k)},\rho)$ on a Hilbert space $\Sc{H}$. Let $\F{A}$ and $\F{S}$ be the unital \textsc{wot}-closed algebra generated by $(A^{(1)},\ldots,A^{(k)},\sigma)$ and $(S^{(1)},\ldots,S^{(k)},\rho)$ respectively. Further, suppose that the representation $(S^{(1)},\ldots,S^{(k)},\rho)$ is doubly commuting.
   Then if $\Sc{V}_1$ is an $\F{A}^*$-invariant subspace of $\Sc{V}$, $\Sc{H}_1=\F{S}[\Sc{V}_1]$ reduces $\F{S}$.

   If $\Sc{V}_1$ and $\Sc{V}_2$ are orthogonal $\F{A}^*$-invariant subspaces the $\Sc{H}_j=\F{S}[\Sc{V}_j]$ for $j=1,2$ are mutually orthogonal.

   If $\Sc{V}=\Sc{V}_1\oplus\Sc{V}_2$, then $\Sc{H}=\Sc{H}_1\oplus\Sc{H}_2$ and $\Sc{H}_j\cap\Sc{V}=\Sc{V}_j$ for $j=1,2$.
  \end{lemma}

  \begin{proof}
   We will prove the first part of the theorem. The remaining parts follow in a similar manner as in Lemma \ref{reducingspace}.

   First, $\Sc{V}_1$ is $\F{A}^*$-invariant, and so $\Sc{V}_1$ is $\F{S}^*$-invariant. Elements of the form $\tilde{S}_\mathbf{n}(\eta\otimes v)$, with $\mathbf{n}\in\B{N}^k$, $\eta\in E(\mathbf{n})$ and $v\in\Sc{V}_1$, span a dense subset of $\Sc{H}_1$. Take $\mathbf{n}=(n_1,\ldots,n_k)\in\B{N}^k$ and $i\in\{1,\ldots,k\}$. Then for any $\xi\in E_i$, $\eta\in E(\mathbf{n})$, $v\in\Sc{V}_1$, $w\in\F{S}[\Sc{V}_1]^\perp$, if $n_i\neq 0$ then
   \begin{align*}
    \<S^{(i)}(\xi)^*\tilde{S}_{\mathbf{n}}(\eta\otimes v),w\>&=\<\tilde{S^{(i)}}^*\tilde{S}_{\mathbf{n}}(\eta\otimes v),\xi\otimes w\>\\
    &=\<I_{E^{i}}\otimes\tilde{S}_{\mathbf{n}-\mathbf{e}_i}(\eta\otimes v),\xi\otimes w\>\\
    &=0,
   \end{align*}
   and so $S^{(i)}(\xi)^*\tilde{S}_{\mathbf{n}}(\eta\otimes v)\in\Sc{H}_1$. If $n_i=0$ then, since our dilation is doubly commuting, 
   \begin{align*}
    \<S^{(i)}(\xi)^*\tilde{S}_{\mathbf{n}}(\eta\otimes v),w\>&=\<\tilde{S^{(i)}}^*\tilde{S}_{\mathbf{n}}(\eta\otimes v),\xi\otimes w\>\\
    &=\<(I_{E_i}\otimes\tilde{S}_{\mathbf{n}})(t\otimes I_\Sc{H})(I_{E(\mathbf{n})}\otimes\tilde{S^{(i)}}^*)(\eta\otimes v),\xi\otimes w\>\\
    &=\<(I_{E_i}\otimes\tilde{S}_{\mathbf{n}})(t\otimes I_\Sc{H})(I_{E(\mathbf{n})}\otimes\tilde{A^{(i)}}^*)(\eta\otimes v),\xi\otimes w\>\\
    &=0
   \end{align*}
   and so again $S^{(i)}(\xi)^*\tilde{S}_{\mathbf{n}}(\eta\otimes v)\in\Sc{H}_1$. Thus $\Sc{H}_1$ is $\F{S}$-reducing.
  \end{proof}

  \begin{remark}
   It is natural to ask if there is a higher rank analogy of Lemma \ref{wanderingspace}. If  $(A^{(1)},\ldots,A^{(k)},\sigma)$ is a representation of $E$ on $\Sc{V}$ with a minimal isometric dilation $(S^{(1)},\ldots,S^{(k)},\rho)$ on $\Sc{H}$, is the restriction of $(S^{(1)},\ldots,S^{(k)},\rho)$ to $\Sc{V}^\perp$ an induced representation? The answer is no. From \cite{Fowler} it is known that induced representations are doubly commuting. Looking at the atomic representations studied in \cite{DavPowYang2} and \cite{DavYang2}, or looking at Example \ref{atomic flip}, we see that the restriction to $\Sc{V}^\perp$ is not, in general, doubly commuting.
  \end{remark}

  \subsection{Finitely Correlated Representations}
  \begin{definition}
    An isometric representation $(S^{(1)},\ldots,S^{(k)},\rho)$ of a product system $E$ on a Hilbert space $\Sc{H}$ is called \emph{finitely correlated} if $(S^{(1)},\ldots,S^{(k)},\rho)$ is the minimal isometric dilation of a representation $(A^{(1)},\ldots,A^{(k)},\sigma)$ on a non-zero finite dimensional Hilbert space $\Sc{V}\subseteq\Sc{H}$.

	In particular, if $\F{S}$ is the unital \textsc{wot}-closed algebra generated by $(S^{(1)},\ldots,S^{(k)},\rho)$, then $(S^{(1)},\ldots,S^{(k)},\rho)$ is finitely correlated if there is a finite dimensional $\F{S}^*$-invariant subspace $\Sc{V}$ of $\Sc{H}$ such that $(S^{(1)},\ldots,S^{(k)},\rho)$ is the minimal isometric dilation of the representation $(P_{\Sc{V}}S^{(1)}(\cdot)|_{\Sc{V}},\ldots,P_{\Sc{V}}S^{(k)}(\cdot)|_{\Sc{V}},\rho(\cdot)|_{\Sc{V}})$.
  \end{definition}

	\begin{remark}
		In this section we are concerned with finitely correlated fully coisometric representations of product systems. Let $(E,\Sc{A})$ be a product system of $C^*$-correspondences over $\B{N}^k$. As in the $C^*$-correspondence case, assuming existence of a finitely correlated fully coisometric representation of $E$ puts restrictions on the $C^*$-algebra $\Sc{A}$. See Remark \ref{finitely representable}. The class of product systems of $C^*$-algebras which exhibit finitely correlated representations includes the $k$-graphs studied in \S\ref{examples}.
	\end{remark}

  A class of finitely correlated representations of $k$-graphs have been studied in \cite{DavPowYang2} ($2$-graphs)  and \cite{DavYang2} ($k$-graphs). These papers consider finitely correlated \emph{atomic} representations of $k$-graphs. These representations are both isometric and fully coisometric. Atomic representations are an example of \emph{partially isometric} representations, i.e. they are representations defined by row-contractions of partial isometries. Atomic representations of $k$-graphs are looked at more closely in \S \ref{Single Vertex Section}. As in the rank $1$ case above, the existence of a unique minimal generating space is shown. We will now prove the existence of such a space for a general finitely correlated, isometric, fully coisometric representation of a product system of $C^*$-correspondences over $\B{N}^k$. We begin with a higher rank version of Lemma \ref{lemma41}.

  \begin{lemma}\label{lemma41 rank k}
	Let $(S^{(1)},\ldots,S^{(k)},\rho)$ be a finitely correlated, fully coisometric representation of a product system $E$ on a Hilbert space $\Sc{H}$. Let $\F{S}$ be the unital \textsc{wot}-closed algebra generated by $(S^{(1)},\ldots,S^{(k)},\rho)$ and let $\Sc{V}$ be a finite dimensional $\F{S}^*$-invariant subspace of $\Sc{H}$ such that $(S^{(1)},\ldots,S^{(k)},\rho)$ is the minimal isometric dilation of the representation $(P_{\Sc{V}}S^{(1)}(\cdot)|_{\Sc{V}},\ldots,P_{\Sc{V}}S^{(k)}(\cdot)|_{\Sc{V}},\rho(\cdot)|_{\Sc{V}})$.

   Then if $\Sc{M}$ is a non-zero, $\F{S}^*$-invariant subspace of $\Sc{H}$, the subspace $\Sc{M}\cap\Sc{V}$ is non-trivial.
  \end{lemma}
  
  \begin{proof}
   Take any $\mathbf{n}=(n_1,\ldots,n_k)\in\B{N}^k$ with $n_i\neq 0$ for $1\leq i\leq k$. By Theorem \ref{rank k to rank 1}, $(S_\mathbf{n},\rho)$ is the unique minimal isometric dilation of $(A_\mathbf{n},\sigma)$. The subspace $\Sc{M}$ is $\F{S}^*$-invariant and so for any $\xi\in E(\mathbf{n})$, $S_\mathbf{n}(\xi)^*\Sc{M}\subseteq\Sc{M}$. Let $\F{S}_\mathbf{n}$ be the unital \textsc{wot}-closed algebra generated by $S_{\mathbf{n}}(E(\mathbf{n}))$ and $\rho(\Sc{A})$. It follows that $\Sc{M}$ is invariant under $\F{S}_\mathbf{n}^*$. Hence, by Lemma \ref{lemma41}, $\Sc{M}\cap\Sc{V}$ is non-trivial.
  \end{proof}

  \begin{theorem}\label{prod sys summary}
	Let $(A^{(1)},\ldots,A^{(k)},\sigma)$ be a fully coisometric representation of a product system $E$ on a finite dimensional Hilbert space $\Sc{V}$, and let $(S^{(1)},\ldots,S^{(k)},\rho)$ be the unique minimal isometric dilation of $(A^{(1)},\ldots,A^{(k)},\sigma)$ on a Hilbert space $\Sc{H}$. Let $\F{A}$ be the unital algebra generated by the representation $(A^{(1)},\ldots,A^{(k)},\sigma)$ and let $\F{S}$ be the unital \textsc{wot}-closed algebra generated by $(S^{(1)},\ldots,S^{(k)},\rho)$.

   If $\Sc{V}_1,\Sc{V}_2,\ldots,\Sc{V}_k$ is a maximal set of pairwise orthogonal minimal $\F{A}^*$-invariant spaces of $\Sc{V}$ then $\hat{\Sc{V}}=\Sc{V}_1\oplus\ldots\oplus\Sc{V}_k$ is the unique minimal cyclic coinvariant subspace of $\Sc{V}$ and $\F{S}^*|_{\hat{\Sc{V}}}$ is a $C^*$-algebra.

   Further, if $\Sc{W}_{\mathbf{m}}$ is the unique, minimal cyclic space for the $C^*$-correspondence representation $(S_{\mathbf{m}},\rho)$, where $\mathbf{m}=(m_1,m_2,\ldots,m_k)$ $($with $m_i\neq 0$ for $1\leq i\leq k)$, then $\Sc{W}_{\mathbf{m}}=\hat{\Sc{V}}$.
  \end{theorem}

  \begin{proof}
   Using Lemma \ref{reducingspace2} and Lemma \ref{lemma41 rank k}, that $\F{S}[\hat{\Sc{V}}]=\Sc{H}$ follows the same argument as in the $C^*$-correspondence case. That $\F{S}^*|_{\hat{\Sc{V}}}$ is a $C^*$-algebra follows by Lemma \ref{cstaralg}.

    Let $\mathbf{m}=(m_1,m_2,\ldots,m_k)\in\B{N}^k$ where $m_i\neq0$ for $1\leq i\leq k$. Let $\F{A}_\mathbf{m}$ be the unital algebra generated by the representation $(A_{\mathbf{m}},\sigma)$ and let $\Sc{W}$ be the unique minimal cyclic coinvariant space for $\F{A}_\mathbf{m}$. By Theorem \ref{rank k to rank 1} and since $\Sc{W}$ is unique, $\Sc{W}$ is contained in any minimal cyclic coinvariant space for $\F{A}$. In particular $\Sc{W}\subseteq\hat{\Sc{V}}$. Note also that $\F{A}[\Sc{W}]=\Sc{V}$ since $\F{A}_\mathbf{m}[\Sc{W}]=\Sc{V}$ and $\F{A}_\mathbf{m}\subseteq\F{A}$. We will show $\Sc{W}$ is $\F{A}^*$-invariant.

    Define the subspace $\Sc{U}'\subseteq\Sc{V}$ by
    \begin{equation*}
     \Sc{U}'=\sum_{\xi\in E(\mathbf{m}-\mathbf{e}_k)}S_{\mathbf{m}-\mathbf{e}_k}(\xi)^*\Sc{W}.
    \end{equation*}
     Note that, by the commutation relations
    \begin{equation*}
     \tilde{S}_{\mathbf{m}-\mathbf{e}_k}(I_{E(\mathbf{m}-\mathbf{e}_k)}\otimes\tilde{S}_\mathbf{m})=\tilde{S}_\mathbf{m}(I_{E(\mathbf{m})}\otimes\tilde{S}_{\mathbf{m}-\mathbf{e}_k})(t\otimes I_\Sc{H})
    \end{equation*}
    where $t$ is the isomorphism $t:E(\mathbf{m}-\mathbf{e}_k)\otimes E(\mathbf{m})\rightarrow E(\mathbf{m})\otimes E(\mathbf{m}-\mathbf{e}_k)$.

    So, if we take vectors $w\in\Sc{W}$, $v\in\Sc{V}\ominus\Sc{U}'$ and $\eta\in E(\mathbf{m})$ and $\xi\in E(\mathbf{m}-\mathbf{e}_k)$ then
    \begin{align*}
     \<S_{\mathbf{m}}(\eta)^*S_{\mathbf{m}-\mathbf{e}_k}(\xi)^*w,v\>
      &=\<(I_{E(\mathbf{m}-\mathbf{e}_k)}\otimes\tilde{S}_{\mathbf{m}}^*)\tilde{S}_{\mathbf{m}-\mathbf{e}_k}^*w,\xi\otimes\eta\otimes v\>\\
      &=\<(I_{E(\mathbf{m})}\otimes\tilde{S}_{\mathbf{m}-\mathbf{e}_k}^*)\tilde{S}_\mathbf{m}^*w,(t\otimes I_\Sc{H})(\xi\otimes\eta\otimes v\>\\
      &=\<\tilde{S}_\mathbf{m}^*w,(I_{E(\mathbf{m})}\otimes\tilde{S}_{\mathbf{m}-\mathbf{e}_k})(t\otimes I_\Sc{H})(\xi\otimes\eta\otimes v\>.
    \end{align*}
    Note that $\tilde{S}_\mathbf{m}^*w\in E(\mathbf{m})\otimes\Sc{W}$, $(I_{E(\mathbf{m})}\otimes\tilde{S}_{\mathbf{m}-\mathbf{e}_k})(t\otimes I_\Sc{H})(\xi\otimes\eta\otimes v)$ is in the space
    \begin{equation*}
     E(\mathbf{m})\otimes\tilde{S}_{\mathbf{m}-\mathbf{e}_k}(E(\mathbf{m}-\mathbf{e}_k)\otimes(\Sc{V}\ominus\Sc{U}')),
    \end{equation*}
    and $\Sc{W}$ and $\Sc{U}'$ are both $\sigma$ reducing subspaces. It follows that 
    \begin{equation*}
     \<S_{\mathbf{m}}(\eta)^*S_{\mathbf{m}-\mathbf{e}_k}(\xi)^*w,v\>=0,
    \end{equation*}
    and so $\Sc{U}'$ is $\F{A}_\mathbf{m}^*$-invariant. By Lemma \ref{lemma41}, $\Sc{U}'$ has non-trivial intersection with $\Sc{W}$. Let $\Sc{U}=\Sc{W}\cap\Sc{U}'$.
  
    Suppose $\Sc{U}\neq\Sc{W}$. A similar argument to above will show that
    \begin{equation*}
     \sum_{\xi\in E_k}S^{(k)}(\xi)^*(\Sc{W}\ominus\Sc{U})
    \end{equation*}
    has non-trivial intersection with $\Sc{W}$. Choose $w_1,\ldots,w_n\in\Sc{W}\ominus\Sc{U}$ and $\zeta_1,\ldots,\zeta_n\in E_k$ such that $\sum_{i=1}^n S^{(k)}(\zeta_i)^*w_i$ is a non-zero vector in $\Sc{W}$. Since $(A^{(1)},\ldots,A^{(k)},\sigma)$ is fully coisometric we can choose $\eta\in E(\mathbf{m}-\mathbf{e}_k)$ such that
    \begin{equation*}
     w:=S_{\mathbf{m}-\mathbf{e}_k}(\eta)^*\sum_{i=1}^n S^{(k)}(\zeta_i)^*w_i
    \end{equation*}
    is non-zero. Now $w$ is in $\Sc{W}$ and hence $w$ is in $\Sc{U}\cap(\Sc{W}\ominus\Sc{U})$. This contradiction shows that we must have $\Sc{U}=\Sc{W}$. By construction of $\Sc{U}'$, we have that for any $u\in\Sc{U}'$ and $\xi\in E_k$, $S^{(k)}(\xi)^*u$ is in $\Sc{W}$. Hence $\Sc{W}$ is invariant under $\F{A}_k^*$, where $\F{A}_j$ is the unital algebra generated by $(A^{(j)},\sigma)$. We can similarly show that $\Sc{W}$ is $\F{A}_j^*$-invariant for $1\leq j\leq k-1$, and so $\Sc{W}$ is $\F{A}^*$-invariant. Therefore $\Sc{W}=\hat{\Sc{V}}$, and thus $\hat{\Sc{V}}$ is unique.
    \end{proof}

   \begin{remark}\label{prod sys summary remark}
     Take a non-zero $\mathbf{m}\in\B{N}^k$ with $\mathbf{m}\not\geq(1,\ldots,1)$. Let $\Sc{U}$ be the minimal cyclic coinvariant subspace for the representation $(A_\mathbf{m},\sigma)$ of the $C^*$-correspondence $E(\mathbf{m})$ and $\hat{\Sc{V}}$ be the minimal cyclic coinvariant subspace for the representation of the product system, as in Theorem \ref{prod sys summary}. We necessarily have that $\Sc{U}\subseteq\hat{\Sc{V}}$. However given an arbitrary finitely correlated representation we can not say whether $\Sc{U}=\hat{\Sc{V}}$ or $\Sc{U}\subsetneq\hat{\Sc{V}}$. For the case when $k=2$ and $\mathbf{m}=(0,1)$, Example \ref{atomic flip} satisfies $\Sc{U}=\hat{\Sc{V}}$ and Example \ref{not partially iso} satisfies $\Sc{U}\subsetneq\hat{\Sc{V}}$.
   \end{remark}

   We again conclude that the compression to the unique minimal cyclic subspace for a finitely correlated, fully coisometric representation is a complete unitary invariant.

  \begin{corollary}\label{complete unitary invariant 2}
   Suppose $(S^{(1)},\ldots,S^{(k)},\sigma)$ and $(T^{(1)},\ldots,T^{(k)},\tau)$ are finitely correlated, fully coisometric representations of a product system $(E,\Sc{A})$ on $\Sc{H}_S$ and $\Sc{H}_T$ respectively. Let $\Sc{V}_S$ be the unique minimal cyclic coinvariant subspace for the representation $(S^{(1)},\ldots,S^{(k)},\sigma)$ and let $\Sc{V}_T$ be the unique minimal cyclic subspace for the representation $(T^{(1)},\ldots,T^{(k)},\tau)$.

   Then $(S^{(1)},\ldots,S^{(k)},\sigma)$ and $(T^{(1)},\ldots,T^{(k)},\tau)$ are unitarily equivalent if and only if the finite dimensional representations $(P_{\Sc{V}_S}S^{(1)}(\cdot)|_{\Sc{V}_S},\ldots,P_{\Sc{V}_S}S^{(k)}(\cdot)|_{\Sc{V}_S},\sigma(\cdot)|_{\Sc{V}_S})$ and $(P_{\Sc{V}_T}T^{(1)}(\cdot)|_{\Sc{V}_T},\ldots,P_{\Sc{V}_T}T^{(k)}(\cdot)|_{\Sc{V}_T},\tau(\cdot)|_{\Sc{V}_T})$ are unitarily equivalent.
  \end{corollary}

  \section{Higher Rank Graph Algebras}\label{examples}
  \subsection{Graph Algebras}\label{graph algebras}
   Let $G$ be a directed graph with a countable number of vertices $\Sc{V}(G)$ and a countable number of edges $\Sc{E}(G)$. If $e\in\Sc{E}(G)$ is an edge from a vertex $v$ to a vertex $w$ then we say that $v$ is the \emph{source} of $e$, denoted $s(e)$, and that $w$ is the \emph{range} of $e$, denoted $r(e)$. A vertex $x$ is called a \emph{source} if there is no edge $e$ with $r(e)=x$. A \emph{path} of length $k$ in $G$ is a finite collection of edges $e_ke_{k-1}\ldots e_1$ such that $r(e_i)=s(e_{i+1})$ for $1\leq i\leq k-1$. A \emph{cycle} is a path $e_ke_{k-1}\ldots e_1$ with $s(e_1)=r(e_k)$. If $x=s(e_1)$ and $y=r(e_k)$ then we say that $e_ke_{k-1}\ldots e_1$ is a path from $x$ to $y$. A graph $G$ is \emph{transitive} if, for any vertices $x,y\in\Sc{V}(G)$, there is a path from $x$ to $y$. A graph is \emph{strongly transitive} if it is transitive and it is neither a single cycle nor a graph with one vertex and no edges.

   As described in \cite{RaeSim, Skalski, MuhlySolel2}  a graph can be described by a $C^*$-correspondence. We follow the construction of \cite{RaeSim} as presented in \cite{Skalski}. Note that in the case of a finite graph this construction is the same as that given in \cite{MuhlySolel2}.

   Let $\Sc{A}=C_0(\Sc{V}(G))$ be the $C^*$-algebra of all functions on $\Sc{V}(G)$ vanishing at infinity. Let $E(G)$ be the set of functions $\xi:\Sc{E}(G)\rightarrow\B{C}$ which satisfy for each $v\in\Sc{V}(G)$
   \begin{equation*}
    \xi_v:=\sum_{\substack{e\in\Sc{E}(G)\\ s(e)=v}}|\xi(e)|^2<\infty
   \end{equation*}
   and the function $v\rightarrow \xi_v$ vanishes at infinity. Define an $\Sc{A}$-valued inner product on $E(G)$ by
   \begin{equation*}
    \<\xi,\eta\>(v)=\sum_{\substack{e\in\Sc{E}(G)\\ s(e)=v}}\overline{\xi(e)}\eta(e),
   \end{equation*}	
   for $\xi,\eta\in E(G)$. Define a left action of $\Sc{A}$ on $E(G)$ by
   \begin{equation*}
    (a\xi)(e)=a(r(e))\xi(e)
   \end{equation*}
   and a right action by
   \begin{equation*}
    (\xi a)(e)=\xi(e)a(s(e))
   \end{equation*}
   for $\xi\in E(G)$, $a\in\Sc{A}$ and $e\in\Sc{E}(G)$. These make $E(G)$ into a $C^*$-correspondence over $\Sc{A}$. We identify the vertex $v\in\Sc{V}(G)$ with function $\delta_v\in\Sc{A}$ which sends $v$ to $1$ and all other vertices to $0$. Similarly, we identify an edge $e\in\Sc{E}(G)$ with the function $\delta_e\in E(G)$ which sends $e$ to $1$ and all other edges to $0$.
   
   For a good introduction to graph algebras see \cite{Rae}. We remark that representations of $E(G)$ coincide with completely contractive representations of $G$ and that the dilation theorem for contractive representations of graphs in \cite{JuryKribs2} and \cite{DavKat} is implied by Theorem \ref{cstar dilation}.

   Denote by $\Sc{L}_G$ the \textsc{wot}-closed algebra generated by
	\begin{equation*}
		\{T_\xi,\varphi_\infty(a):\xi\in E,a\in\Sc{A}\}
	\end{equation*}
	 acting on the space $\Sc{H}_G:=\Sc{F}(E(G))$. The algebra $\Sc{L}_G$ is known as a \emph{free semigroupoid algebra}, see \cite{KribsPower1}.  When $G$ has a single vertex and $n$ edges then $\Sc{L}_G$ is a \emph{free semigroup algebra}, more commonly denoted $\Sc{L}_n$. 

   Finite dimensional representations of graphs are plentiful. Indeed Davidson and Katsoulis show that the finite dimensional representations of a graph $G$ separate points in $\Sc{L}_G$, \cite{DavKat}. Thus finitely correlated, isometric representations are also plentiful. Provided in \cite{DavKat} is an algorithm for creating finite dimensional representations. Below is a method for creating finite dimensional, fully coisometric representations. A similar example can be found in \cite{JuryKribs2}.

   \begin{example}
    Let $G$ be a finite graph with no sources. Let $\Sc{V}(G)=\{v_1,\ldots,v_n\}$. Let $\Sc{E}(G)_i=\{e\in\Sc{E}(G):r(e)=v_i\}=\{e_{i1},e_{i2},\ldots,e_{iC_i}\}$, where $C_i$ is the number of elements in $\Sc{E}(G)_i$. Let $\Sc{A}$ and $E(G)$ be as described above. Let $\Sc{H}$ be a finite dimensional Hilbert space and let $\Sc{K}=\Sc{H}_1\oplus\ldots\oplus\Sc{H}_n$, where $\Sc{H}_i=\Sc{H}$ for each $i$. We will define a representation $(A,\sigma)$ of $E(G)$ on $\Sc{K}$.
    
    For each vertex $v_i$ let $T_i=[T_{i1},\ldots,T_{iC_i}]$ be a defect free row-contraction on $\Sc{H}_i$, i.e.
    \begin{equation*}
     \sum_{j=1}^{C_i}T_{ij}T_{ij}^*=I_{\Sc{H}_i}.
    \end{equation*}
    Suppose $e_{ij}\in\Sc{E}(G)_i$ with $s(e_{ij})=v_l$. Define $A(e_{ij})\in\Sc{B}(\Sc{K})=M_{n}(\Sc{B}(\Sc{H}))$ by $(A(e_{ij}))_{i,l}=T_{i,j}$ and $(A(e_{ij}))_{k,m}=0$ when $(k,m)\neq(i,l)$. We define a representation $\sigma$ of $\Sc{A}$ on $\Sc{K}$ by $\sigma(v_i)=P_{\Sc{H}_i}=:P_{v_i}$ for $1\leq i\leq n$. Thus
    \begin{equation*}
     \sum_{e\in\Sc{E}(G)}A(e)A(e)^*=I_{\Sc{K}}
    \end{equation*}
    and
    \begin{equation*}
     P_{r(e)}A(e)P_{s(e)}=A(e).
    \end{equation*}
    It follows that $(A,\sigma)$ is a finite dimensional, fully coisometric representation of $E(G)$, see \cite{DavKat}, \cite{JuryKribs2}. This method readily extends to any graph containing a finite subgraph with no sources.
   \end{example}

  \subsubsection{Strong Double-Cycle Property}
  We now strengthen our results from \S \ref{cstar cor} for the special case of $C^*$-correspondences defined by finite graphs with the strong double-cycle property.
  \begin{definition}
   A vertex in $G$ is said to lie on a \emph{double-cycle} if it lies on two, distinct, minimal cycles. We say that $G$ has the \emph{strong double-cycle property} if for every vertex $x$ in $G$ there is a path from $x$ to a vertex lying on a double-cycle.
  \end{definition}

  \begin{example}
   When $n\geq 2$, a single vertex graph with $n$ edges has the strong double-cycle property. This is the case studied in \cite{DavKriShp}.
  \end{example}

  \begin{example}
   If each connected component of $G$ is strongly transitive, then $G$ has the strong double-cycle property.
  \end{example}

  The following result is proved in \cite{MuhlySolel2} and \cite{KribsPower1} for finite graphs with the strong double-cycle property, and in \cite{DavPitts} for when $\Sc{L}_G=\Sc{L}_n$ is a free semigroup algebra.

  \begin{theorem}\label{A1doublecycle}
   Suppose $G$ is a finite graph with the strong double-cycle property and $\varphi$ is a weak-$*$ continuous linear functional on $\Sc{L}_G$ with $\|f\|<1$. Then there are vectors $\xi$ and $\zeta$ in $\Sc{H}_G$, with $\|\xi\|,\|\zeta\|<1$, such that $\varphi(A)=\<A\xi,\zeta\>$ for all $A$ in $\Sc{L}_G$.
  \end{theorem}

  We fix such a graph $G$ with a finitely correlated fully coisometric representation $(S,\rho)$ of $E(G)$ on a Hilbert space $\Sc{H}$. Let $\Sc{U}$ be the unique minimal cyclic coinvariant subspace for $(S,\rho)$ and let $(A,\sigma)$ be the compression of $(S,\rho)$ to $\Sc{U}$, so that $(S,\rho)$ is the unique minimal dilation of $(A,\sigma)$. Let $\F{A}$ be the unital algebra generated by $(A,\sigma)$ and let $\F{S}$ be the unital \textsc{wot}-closed algebra generated by $(S,\rho)$. By Theorem \ref{summary}, $\Sc{U}=\Sc{U}_1\oplus\ldots\oplus\Sc{U}_n$ is a direct sum of minimal $\F{A}^*$-invariant subspaces and $\F{A}$ is a $C^*$-algebra. For each $j$ let $\Sc{H}_j=\F{S}[\Sc{U}_j]$. Let $d=\dim\Sc{U}$ and let $\{f_1,\ldots,f_d\}$ form an orthonormal basis of $\Sc{U}$. We now follow the methods in \cite{DavKriShp} in order to give a full description of $\F{S}$. In particular, we will show that $\F{S}$ contains the projection onto $\Sc{U}$.

  For $1\leq i\leq n$, let $q_i$ be the compression of $\F{A}$ to $\Sc{U}_i$, i.e. $q_i(\F{A})=P_{\Sc{U}_i}\F{A}P_{\Sc{U}_i}=\Sc{B}(\Sc{U}_i)$. Choose a minimal set $H\subseteq\{1,\ldots,n\}$ such that $\sum_{h\in H}^\oplus q_h$ is faithful. The minimal ideal $\ker\sum_{h\in H\backslash\{h_0\}}^\oplus q_h$ is isomorphic to $\Sc{B}(\Sc{U}_{h_0})$. This kernel can be supported on more than one of the $\Sc{U}_i$'s. We let $H_h\subseteq H$ be the set of indices $i$ where $\Sc{U}_i$ is supported on $\ker\sum_{g\in H\backslash\{h\}}^\oplus q_g$. For each $h\in H$ let $m_h$ be the number of elements in $H_h$. If we let $\Sc{W}_h=\sum_{i\in H_h}^\oplus\Sc{U}_i$, then $\Sc{U}=\sum_{h\in H}^\oplus\Sc{W}_h$. For each $j\in H_h$ there is a spatial, algebra isomorphism $\sigma_j$ of $\Sc{B}(\Sc{U}_h)$ onto $\Sc{B}(\Sc{U}_j)$ such that
  \begin{equation*}
   \F{A}|_{\Sc{W}_h}=\Biggl\{\sideset{}{^\oplus}\sum_{j\in H_h}\sigma_j(X):X\in\Sc{B}(\Sc{U}_h)\Biggr\}.
  \end{equation*}
  For each $h\in H$ let $P_h$ be the projection onto $\Sc{W}_h$. For each $h\in H$ the projection $P_h$ lies in the centre of $\F{A}$.

  A closer look at Lemma \ref{wanderingspace} tells us that for each $\xi\in E$ and $a\in\Sc{A}$
  \begin{equation*}
   S(\xi)=\begin{bmatrix}A(\xi)&0\\X_\xi&T_\xi^{(\alpha)}\end{bmatrix},\ \ \rho(a)=\begin{bmatrix}\sigma(a)&0\\0&\rho(a)|_{\Sc{V}^\perp}\end{bmatrix}
  \end{equation*}
  where $\alpha=\dim\Sc{W}$, with $\Sc{W}=(\Sc{U}+\tilde{S}(E\otimes\Sc{U}))\ominus\Sc{U}$ as in Lemma \ref{wanderingspace}. Hence
  \begin{equation*}
   \F{S}=\begin{bmatrix}\F{A}&0\\ * &\Sc{L}_G^{(\alpha)}\end{bmatrix}.
  \end{equation*}

  We denote by $\F{B}$ the \textsc{wot}-closed operator algebra on $\Sc{H}$ spanned by $\Sc{B}(\Sc{H})P_\Sc{U}$ and $0_\Sc{U}\oplus\Sc{L}_G^{(\alpha)}$. The following three proofs follow the arguments of \cite[Lemma 4.4]{DavKriShp}, \cite[Lemma 5.14]{DavKriShp} and \cite[Corollary 5.3]{DavKriShp} respectively.

  \begin{lemma}\label{DKS Lemma 4.4}
   Every weak-$*$ continuous functional on $\F{B}$ is given by a trace class operator of rank at most $d+1$, where $d=\dim\Sc{U}$.

   Hence the \textsc{wot} and weak-$*$ topologies coincide on $\F{B}$ and $\F{S}$.
  \end{lemma}

  \begin{proof}
   Let $\varphi$ be a weak-$*$ continuous functional on $\F{B}$. If $B\in\F{B}$ then $\varphi(B)$ is determined by $\varphi(BP_\Sc{U})$ and $\varphi(BP_{\Sc{U}^\perp})$. By the Riesz Representation Theorem there are vectors $y_1,\ldots,y_d\in\Sc{U}$ such that
   \begin{equation*}
    \varphi(BP_\Sc{U})=\sum_{i=1}^d\<Bf_i,y_i\>.
   \end{equation*}
   By Theorem \ref{A1doublecycle} there are vectors $\xi,\zeta\in\Sc{U}^\perp$ such that $\varphi(A)=\<A\xi,\zeta\>$ for all $A\in\Sc{L}_G^{(\alpha)}$. Hence 
   \begin{equation*}
    \varphi(B)=\sum_{i=1}^d\<Bf_i,y_i\>+\<B\xi,\zeta\>,
   \end{equation*}
   and $\varphi$ is trace-class of rank at most $d+1$.
  \end{proof}

  \begin{lemma}\label{central projections}
   For $h\in H$, let $P_h$ denote the minimal central projections of $\F{A}$ as above. Then $P_h$ lies in $\F{S}$. Hence $P_{\Sc{U}}$ is in $\F{S}$.
  \end{lemma}

  \begin{proof}
   Fix a minimal central projection $P$ of $\F{A}$. Let $\varphi$ be a non-zero weak-$*$ continuous functional on $\F{B}$ which is zero on $\F{S}$. We will show that $\varphi(P)=0$. It follows immediately that $P\in\F{S}$.

   By Lemma \ref{DKS Lemma 4.4} there are vectors $x,y\in\Sc{H}^{(d+1)}$ such that $\varphi(A)=\<A^{(d+1)}x,y\>$ for all $A\in\F{B}$. Let $\Sc{M}=\F{S}^{*(d+1)}[y]$. Since $\varphi$ is zero on $\F{S}$ it follows that $x$ is orthogonal to $\Sc{M}$. Let $\Sc{M}_0=\Sc{M}\cap\Sc{U}^{(d+1)}$. By Lemma \ref{lemma41}, $\Sc{M}_0$ is non-zero. The subspace $\Sc{M}_0$ is invariant under the $C^*$-algebra $\F{A}^{(d+1)}$ and hence $\Sc{M}_0$ is the range of a projection $Q$ in the commutant of $\F{A}^{(d+1)}$.

   We decompose $\Sc{U}^{(d+1)}$ into the following spaces
   \begin{align*}
    &P^{(d+1)}Q\Sc{U}^{(d+1)}\oplus P^{\perp(d+1)}Q\Sc{U}^{(d+1)}\oplus P^{(d+1)}Q^\perp\Sc{U}^{(d+1)}\oplus P^{\perp(d+1)}Q^\perp\Sc{U}^{(d+1)}\\
    &=:\Sc{M}_{pq}\oplus\Sc{M}_{p^\perp q}\oplus\Sc{M}_{pq^\perp}\oplus\Sc{M}_{p^\perp q^\perp}.
   \end{align*}
   Note that, as $Q$ and $P^{(d+1)}$ are projections in the commutant of $\F{A}^{(d+1)}$, the four spaces $M_{ij}$ are $\F{A}^{(d+1)}$-reducing. Also $\Sc{M}_0=\Sc{M}_{pq}\oplus\Sc{M}_{p^\perp q}$. Letting $\Sc{H}_{ij}=\F{S}[\Sc{M}_{ij}]$ we see that $\Sc{H}$ decomposes into
   \begin{equation*}
    \Sc{H}=\Sc{H}_{pq}\oplus\Sc{H}_{p^\perp q}\oplus\Sc{H}_{p q^\perp}\oplus\Sc{H}_{p^\perp q^\perp}.
   \end{equation*}
   It follows that $y\in\Sc{H}_{pq}\oplus\Sc{H}_{p^\perp q}=\F{S}[\Sc{M}_0]$ and so $P_{\Sc{U}}^{(d+1)}y\in\Sc{M}_0$. The projection $P_\Sc{U}$ dominates $P$ and so $P^{(d+1)}y\in\Sc{M}_0$. Hence
   \begin{equation*}
    \varphi(P)=\<P^{(d+1)}x,y\>=\<x,P^{(d+1)}y\>=0.
   \end{equation*}
  \end{proof}

  \begin{lemma}\label{central projections 2}
   The algebra $\F{S}P_\Sc{U}\simeq\sum_{h\in H}^\oplus(\Sc{B}(\Sc{H}_h)P_h)^{(m_h)}$, where $m_h=|H_h|$.
  \end{lemma}

  \begin{proof}
   First suppose that $\F{A}=\Sc{B}(\Sc{U})$, i.e. $\Sc{U}$ is a minimal $\F{A}^*$-invariant subspace. By Lemma \ref{central projections}, the projection $P_{\Sc{U}}$ is in $\F{S}$. Hence $\F{S}P_{\Sc{U}}=\Sc{B}(\Sc{U})$ is in $\F{S}$. In particular, for any $v\in\Sc{U}$ the rank $1$ operator $vv^*$ is in $\F{S}$. Note also that $\F{S}[v]=\Sc{H}$ for any non-zero $v\in\Sc{U}$. Hence for any $x\in\Sc{H}$ there are operators $T_k$ in $\F{S}$ such that $T_kv$ converges to $x$. Hence $T_kvv^*$ is in $\F{S}$. Hence $xv^*$ is in $\F{S}$ for all $x\in\Sc{H}$ and $v\in\Sc{U}$. Therefore $\Sc{B}(\Sc{H})P_{\Sc{U}}$ is in $\F{S}$.

   Returning to the general case, note that there is a unitary equivalence between $\sum_{j\in H_h}^\oplus\Sc{H}_j$ and $\Sc{H}_h\otimes\B{C}^{(m_h)}$. Lemma \ref{central projections} tells us that $P_{\Sc{W}_h}\simeq P_h^{(m_h)}$ lies in $\F{S}$ for each $h\in H$. From the first paragraph it now follows that $\F{S}P_{\Sc{U}}$ decomposes as $\sum_{h\in H}^\oplus(\Sc{B}(\Sc{H}_h)P_h)^{(m_h)}$.
  \end{proof}

  Combining Lemma \ref{central projections} and Lemma \ref{central projections 2} with Theorem \ref{summary} we get the following theorem. When $G$ is a single vertex graph with $2$ or more edges, Theorem \ref{graph summary} is the same as \cite[Theorem 5.15]{DavKriShp}.

  \begin{theorem}\label{graph summary}
   Let $G$ be a finite graph with the strong double cycle property. Let $(A,\sigma)$ be fully coisometric, finite dimensional representation of $G$ on a Hilbert space $\Sc{U}$. Let $(S,\rho)$ be the unique minimal isometric dilation of $(A,\sigma)$ to a Hilbert space $\Sc{K}$. Let $\F{A}=\Alg\{A(\xi),\sigma(a):\xi\in E,a\in\Sc{A}\}$ and $\F{S}=\Alg\{S(\xi),\rho(a):\xi\in E,a\in\overline{\Sc{A}\}}^{\textsc{wot}}$

   If
   \begin{equation*}
    \hat{\Sc{U}}=\sideset{}{^\oplus}\sum_{j=1}^n\Sc{U}_j
   \end{equation*}
   is a maximal direct sum of minimal $\F{A}^*$-invariant subspaces of $\Sc{U}$ then $\hat{\Sc{U}}$ is the \emph{unique} minimal $\F{A}^*$-invariant 	subspace such that $\F{S}[\hat{\Sc{U}}]=\Sc{H}$. The compression $\hat{\F{A}}$ of $\F{A}$ to $\hat{\Sc{U}}$ is a $C^*$-algebra. Writing $\hat{\Sc{U}}$ as $\sum_{h\in H}^\oplus\Sc{U}_h^{(m_h)}$, where $\Sc{U}_h$ has dimension $d_h$ and multiplicity $m_h$ then
   \begin{equation*}
    \hat{\F{A}}=\sideset{}{^\oplus}\sum_{h\in h}M_{d_h}\otimes\B{C}^{m_h}.
   \end{equation*}
   Let $P_h$ be the projection onto $\Sc{U}_h$. Then the dilation acts on the space
   \begin{equation*}
    \Sc{K}=\sideset{}{^\oplus}\sum_{h\in H}\Sc{K}_h^{(m_h)}=\hat{\Sc{U}}\oplus\Sc{H}_G^{(\alpha)}
   \end{equation*}
   where $\Sc{K}_h=\Sc{U}_h\oplus\Sc{H}_h^{(\alpha_h)}$, $\alpha_h=\dim((\tilde{S}(E(G)\otimes\Sc{U}_h)\ominus\Sc{U}_h)$ and
   \begin{equation*}
    \alpha=\sum_{h\in H}\alpha_hm_h.
   \end{equation*}
   The algebra $\F{S}$ decomposes as
   \begin{equation*}
    \F{S}\simeq\sideset{}{^\oplus}\sum_{h\in H}(\Sc{B}(\Sc{H}_h)P_h)^{(m_h)}+(0_{\hat{\Sc{U}}}\oplus\Sc{L}_G^{(\alpha)}).
   \end{equation*}
  \end{theorem}

  \subsection{Higher Rank Graph Algebras}\label{k graphs}
  %Higher-rank graphs were first introduced by Kumjian and Pask in \cite{KumPask}. The study of the nonself-adjoint algebras generated by representations of higher-rank graphs was first studied in \cite{KribsPower3}. Further work in the nonself-adjoint case was carried out in \cite{Power} and \cite{SkalZach2}.

  \begin{definition}
   A \emph{$k$-graph} $(\Lambda,d)$ consists of a countable small category $\Lambda$, together with a degree functor $d$ from $\Lambda$ to $\B{N}^k$, satisfying the factorization property: for every $\lambda\in\Lambda$ and $\mathbf{m},\mathbf{n}\in\B{N}^k$ with $d(\lambda)=\mathbf{m}+\mathbf{n}$, there are unique elements $\mu, \nu\in\Lambda$ such that $\lambda=\mu\nu$ and $d(\mu)=\mathbf{m}$ and $d(\nu)=\mathbf{n}$. For each $\mathbf{n}\in\B{N}^k$ let $\Lambda^{\mathbf{n}}=d^{-1}(\mathbf{n})$. Each $k$-graph $(\Lambda,d)$ has a \emph{source} map $s:\Lambda\rightarrow\Lambda^0$ and a \emph{range} map $r:\Lambda\rightarrow\Lambda^0$.

   A $k$-graph $\Lambda$ is said to be \emph{finitely aligned} if for each $\lambda,\mu\in\Lambda$ the set $\{\nu\in\Lambda:\exists_{\alpha,\beta\in\Lambda}\ \nu=\lambda\alpha=\mu\beta,\ d(\nu)=d(\lambda)\vee d(\mu)\}$, is finite.
  \end{definition}

  A $1$-graph $(\Lambda,d)$ is simply a graph with vertices $\Lambda^0$ and edges $\Lambda^1$. A $k$-graph can be visualized as a multi-coloured graph with vertices $\Lambda^0$ and $\Lambda^{\mathbf{e}_i}$ representing a different coloured set of edges for each $i$.

  As in the $1$-graph case, a $k$-graph can be associated with a product system of $C^*$-correspondences over $\B{N}^k$. Briefly, define a $C^*$-algebra $\Sc{A}$ by $\Lambda^0$, in the same manner that we used the vertices of a $1$-graph to define a $C^*$-algebra. For $1\leq i\leq k$ define a $C^*$-correspondence $E_i$ over $\Sc{A}$ by $\Lambda^{\mathbf{e}_i}$ in the same manner that we defined a $C^*$-correspondence using the edges of a $1$-graph. The factorisation rule of $(\Lambda,d)$ will define the isomorphisms $t_{i,j}:E_i\otimes E_j\rightarrow E_j\otimes E_i$, and this in turn will define a product system of $C^*$-correspondences $(E(\Lambda),\Sc{A})$ over $\B{N}^k$, see \cite{RaeSim} or \cite{Skalski} for the details. In \cite{RaeSim} it is shown that Toeplitz $\Lambda$-families of contractions coincide with isometric representations of $E(\Lambda)$. In \cite{Skalski} it is shown that $\Lambda$-contractions coincide with representations of $E(\Lambda)$. Thus there is a $1-1$ correspondence between representations of the $k$-graph $(\Lambda,d)$ and representations of $(E(\Lambda),\Sc{A})$.

  When $\Lambda$ is finitely aligned then $E(\Lambda)$ will satisfy the normal ordering condition. Hence Theorem \ref{SkalskiDil} can be applied to finitely aligned $k$-graphs. This is the dilation theorem originally proved in \cite{SkalZach2}. 

	Let $\Lambda$ be a $k$-graph with no sources and with $\Lambda^0$ finite. In \cite[Theorem 4.7]{SkalZach2} it is shown that there is a $1-1$ correspondence between states $\omega$ on the Cuntz-Pimsner algebra $\Sc{O}_\Lambda$ and (the unitary equivalence classes of) triples $(\Sc{V},\Omega,(S^{(1)},\ldots,S^{(k)},\rho))$ where $\Sc{V}$ is a vector space, $\Omega\in\Sc{V}$ is norm $1$ vector, $(S^{(1)},\ldots,S^{(k)},\rho)$ is an isometric representation of $E(\Lambda)$, and $\Sc{V}=\overline{\F{S}^*\Omega}$ (where $\F{S}$ is the algebra generated by $(S^{(1)},\ldots,S^{(k)},\rho)$). It is noted in \cite{SkalZach2} that $(S^{(1)},\ldots,S^{(k)},\rho)$ is the minimal isometric dilation of the compression of $(S^{(1)},\ldots,S^{(k)},\rho)$ to $\Sc{V}$. Given this result, it is natural to define what it means for a state on $\Sc{O}_\Lambda$ to be finitely correlated as follows:

\begin{definition}A state $\omega$ on $\Sc{O}_\Lambda$ is \emph{finitely correlated} if its corresponding triple $(\Sc{V}_\omega,\Omega_\omega,(S_\omega^{(1)},\ldots,S_\omega^{(k)},\rho_\omega))$ has the property that $\Sc{V}_\omega$ is finite dimensional. 
\end{definition}
When $\omega$ is a finitely correlated state on the Cuntz-Pimsner algebra $\Sc{O}_\Lambda$ with corresponding triple $(\Sc{V}_\omega,\Omega_\omega,(S_\omega^{(1)},\ldots,S_\omega^{(k)},\rho_\omega))$, the representation $(S_\omega^{(1)},\ldots,S_\omega^{(k)},\rho_\omega)$ will be finitely correlated. When $\Lambda$ is a $1$-graph with a single vertex and $n$ edges, $\Sc{O}_\Lambda$ is the Cuntz algebra $\Sc{O}_n$ and the above definition coincides with the definition of finitely correlated states in \cite{BrattJorg1}.

  Theorem \ref{prod sys summary} and Theorem \ref{graph summary} together give us the following result.

  \begin{proposition}\label{higher rank proj}
   Let $(\Lambda, d)$ be a $k$-graph. Suppose there is an $\mathbf{n}=(n_1,\ldots,n_k)\in\B{N}^k$ with $n_i\neq0$ for $1\leq i\leq k$, such that the $1$-graph with vertices $\Lambda^0$ and edges defined by $\Lambda^{\mathbf{n}}$ has the strong double-cycle property. Let $(S^{(1)},\ldots,S^{(k)},\rho)$ be a finitely correlated, isometric, fully coisometric representation of $E(\Lambda)$ generating a \textsc{wot}-closed algebra $\F{S}$. Then $\F{S}$ contains the projection onto its minimal cyclic coinvariant subspace.
  \end{proposition}

  \subsubsection{Graphs With a Single Vertex}\label{Single Vertex Section}
  Suppose $(\Lambda,d)$ is a $k$-graph where $\Lambda^0$ is a singleton and $\Lambda^{\mathbf{e}_i}$ is finite for $1\leq i\leq k$. Let $\Lambda^{\mathbf{e}_i}=\{e_l^{(i)}:1\leq l\leq m_i\}$, where $m_i$ is the number of elements in $\Lambda^{\mathbf{e}_i}$. Let $S_{m_i\times m_j}$ be the set of permuations on the set of tuples $\{(a,b):1\leq a\leq m_i,\ 1\leq b\leq m_j\}$. By the factorisation property, for each pair $i,j$ with $1\leq i<j\leq k$ there is a permutation $\theta_{ij}\in S_{m_i\times m_j}$ such that
  \begin{equation*}
   e_l^{(i)}e_m^{(j)}=e_{m'}^{(j)}e_{l'}^{(i)}
  \end{equation*}
  when $\theta_{ij}(l,m)=(l',m')$. Let $\theta=\{\theta_{ij}:1\leq i<j\leq k\}$. The $k$-graph $\Lambda$ can be described as being a unital semigroup $\B{F}_\theta^+$, where $\B{F}_\theta^+$ is the semigroup
  \begin{equation*}
   \<e_l^{(i)}:e_l^{(i)}e_m^{(j)}=e_{m'}^{(j)}e_{l'}^{(i)}\text{ when }\theta_{ij}(l,m)=(l',m')\>.
  \end{equation*}	
  That is, for each $i$, $e_i^{(1)},\ldots, e_i^{(m_i)}$ form a copy of the free semigroup $\B{F}_{m_i}^+$ and, when $i\neq j$ and $i<j$, a commutation relation between the $e_i$'s and the $e_j$'s is defined by the permutation $\theta_{ij}$. Note that if we are given arbitrary permutations $\theta_{ij}\in S_{m_i\times m_j}$ for $1\leq i<j\leq k$ we cannot necessarily form a cancellative semigroup $\B{F}_\theta^+$. However, if $k=2$ and $\theta\in S_{m_1\times m_2}$ is any permutation, $\B{F}_\theta^+$ will form a cancellative semigroup, and hence a $2$-graph on a single vertex.

  Let $(E(\B{F}_\theta^+),\Sc{A})$ be the product system of $C^*$-correspondences defined by a $k$-graph on a single vertex $\B{F}_\theta^+$. It is not hard to see that $\Sc{A}=\B{C}$ and that $E_i=\B{C}^{m_i}$ for $1\leq i\leq m$. Let $(A^{(1)},\ldots,A^{(k)},\sigma)$ be a representation of $(E(\B{F}_\theta^+),\Sc{A})$ on a Hilbert space $\Sc{H}$ and define $A_l^{(i)}=A^{(i)}(e_l^{(i)})$. For each $i$ we have that
  \begin{equation*}
   A^{(i)}=[A_1^{(i)},\ldots,A_{m_i}^{(i)}]
  \end{equation*}
  is a row-contraction. A representation $(A^{(1)},\ldots,A^{(k)},\sigma)$ is fully coisometric when
  \begin{equation}
   \sum_{j=1}^{m_i}A_j^{(i)}A_j^{(i)*}=I_\Sc{H},
  \end{equation}
  for $1\leq i\leq k$ i.e. when each row-contraction is defect free. A representation is isometric when $[A_1^{(i)},\ldots,A_{m_i}^{(i)}]$ is a row-isometry for $1\leq i\leq k$.

  Conversely, if $[A_1^{(i)},\ldots,A_{m_i}^{(i)}]$ are $k$ row-contractions which satisfy for $1\leq i<j\leq k$
  \begin{equation*}
   A_l^{(i)}A_m^{(j)}=A_{m'}^{(j)}A_{l'}^{(i)}
  \end{equation*}
  when $\theta_{ij}(l,m)=(l',m')$, then they define a representation of the $k$-graph $\B{F}_\theta^+$.

  The $k$-graph $\B{F}_\theta^+$ is finite and so it is finitely aligned. Thus, by either Theorem \ref{coisometric dilations} or Theorem \ref{SkalskiDil} together with Lemma \ref{fullycoiso uniq dil}, all fully coisometric representations of $\B{F}_\theta^+$ have a unique minimal isometric, coisometric dilation.

  Let $\B{F}^+$ be the unital free semigroup with $m_1m_2\ldots m_k$ generators 
  \begin{equation*}
   \{e_{l_1}^{(1)}e_{l_2}^{(2)}\ldots e_{l_k}^{(k)}:1\leq l_j\leq m_j\}.
  \end{equation*}
  This corresponds to the graph with $1$-vertex and $C^*$-correspondence $E(1,1,\ldots,1)$. If $m_1\ldots m_k\neq 1$, i.e. if $\B{F}^+\not\cong\B{Z}_{\geq0}$, then it is clear that $\B{F}^+$ has the strong double-cycle property. Thus by Proposition \ref{higher rank proj}, if $[S_1^{(i)},\ldots,S_{m_i}^{(i)}]$ are defect free row-isometries defining a finitely correlated representation of $\B{F}_\theta^+$, then the \textsc{wot}-closed algebra they generate contains the projection onto the minimal cyclic coinvariant subspace.

  \begin{definition}
   Let $[A_1^{(i)},\ldots,A_{m_i}^{(i)}]$, for $1\leq i\leq k$, define a representation of $\B{F}_\theta^+$ on a Hilbert space $\Sc{H}$. The representation is \emph{atomic} if each $A_l^{(i)}$ is a partial isometry and there is an orthonormal basis $\{\xi_n:n\geq 1\}$ of $\Sc{H}$ which is permuted, up to scalars, by each partial isometry, i.e. $A_l^{(i)}\xi_n=\alpha\xi_m$ for some $m$ and some $\alpha\in\B{T}\cup\{0\}$.
  \end{definition}

  Atomic representations of $k$-graphs on a single vertex were studied by Davidson, Power and Yang for $2$-graphs \cite{DavPowYang2} and by Davidson and Yang for $k$-graphs \cite{DavYang2}. There the existence of the minimal cyclic coinvariant subspace is shown. The minimal cyclic coinvariant subspace for a finitely correlated, isometric, fully coisometric atomic representation is exhibited by a group construction. That is, a finitely correlated, isometric, fully coisometric atomic representation is shown to be a dilation of a certain representation on $\Sc{B}(\ell^2(G))$ where $G$ is a group with $k$ generators. The following theorem shows that finitely correlated atomic representations are plentiful.
  
  \begin{theorem}[Davidson, Power and Yang \cite{DavPowYang2, DavYang2}]
   There are irreducible finite dimensional defect free atomic representations of $\B{F}_\theta^+$ of arbitrarily large dimension.
  \end{theorem}

  \begin{example}\label{atomic flip}
   Let $\B{F}_\theta^+$ be the two graph where $\theta\in S_{2\times 2}$ is the permutation defined by the cycle $((1,1),(2,2))$. Let $\Sc{V}$ be a $4$ dimensional vector space with orthonormal basis $\{\zeta_1,\zeta_2,\zeta_3,\zeta_4\}$. We define a fully coisometric, atomic representation of $\B{F}_\theta^+$ on $\Sc{V}$ by row-contractions $[A_1,A_2]$ and $[B_1,B_2]$, where
   \begin{align*}
    &A_1\zeta_1=\zeta_2& A_1\zeta_3=\zeta_4&& A_1\zeta_i&=0 \ \text{ for }\ i=2,4\\
    &A_2\zeta_2=\zeta_1& A_2\zeta_4=\zeta_3&& A_1\zeta_i&=0 \ \text{ for }\ i=1,3\\
    &B_1\zeta_2=\zeta_3& B_1\zeta_4=\zeta_1&& B_1\zeta_i&=0 \ \text{ for }\ i=1,3\\
    &B_2\zeta_1=\zeta_4& B_2\zeta_3=\zeta_2&& B_1\zeta_i&=0 \ \text{ for }\ i=2,4.
   \end{align*}
   Let $[S_1,S_2]$ and $[T_1,T_2]$ define the unique minimal isometric dilation of this representation. The representation defined by $[S_1,S_2]$ and $[T_1,T_2]$ will also be atomic \cite{DavPowYang1}. Clearly $\Sc{V}$ is the minimal cyclic coinvariant subspace for this representation. For $u,w\in\B{F}_2^+$, where $u=i_1\ldots i_l$ and $w=j_i\ldots j_m$, we write $S_uT_w$ for
   \begin{equation*}
    S_{i_1}\ldots S_{i_l}T_{j_1}\ldots T_{j_m}.
   \end{equation*}
   The set $\{S_u T_w\zeta_i:u,w\in\B{F}_2^+, i=1,2,3,4\}$ will form an orthonormal basis of $\Sc{H}$. Since the representation is atomic and fully coisometric each of these basis vectors will be in the range of exactly one $S_i$ and exactly one $T_j$. It follows that $[S_1,\ldots,S_n]$ is the minimal isometric Frazho-Bunce-Popescu dilation of the row-contraction $[A_1,\ldots,A_n]$. That is, in this case, it is not necessary to have $\mathbf{m}\geq(1,1)$ in order for the conclusion of Theorem \ref{rank k to rank 1} to be satisfied. This is true of all finitely correlated atomic representations. Recall, by Remark \ref{rank k to rank 1 remark}, that in general we do need the condition that $\mathbf{m}\geq(1,1)$ for Theorem \ref{rank k to rank 1} to hold.

   We also have that the minimal cyclic coinvariant subspace for $[S_1,\ldots,S_n]$ is all of $\Sc{V}$. Thus, again, it is not necessary to have $\mathbf{m}\geq(1,1)$ for the conclusion of Theorem \ref{prod sys summary} to be satisfied. This is also a general fact about atomic representations. Again, recall that we do require that $\mathbf{m}\geq(1,1)$ in the general case. See Remark \ref{prod sys summary remark} and the following example.
   \end{example}

  There are examples of finite dimensional, fully coisometric representations which are not partially isometric.
  \begin{example}\label{not partially iso}
   Let $\theta\in S_{2\times 2}$ be the permutation defined by $\theta(1,1)=(1,2)$, $\theta(1,2)=(1,1)$, $\theta(2,1)=(2,2)$ and $\theta(2,2)=(2,1)$, and let $\B{F}_\theta^+$ be the single vertex $2$-graph defined by $\theta$. Let $[a_1,a_2]$ be a defect free row-contraction on a finite dimensional Hilbert space $\Sc{V}$ and $[b_1,b_2]$ be a defect free row-contraction on a finite dimensional Hilbert space $\Sc{W}$. We will define a representation of $\B{F}_\theta^+$ on $\Sc{V}\otimes\Sc{W}^{(2)}$. Define
   \begin{align*}
      A_1& =a_1\otimes\begin{bmatrix}0&I_\Sc{W}\\I_\Sc{W}&0\end{bmatrix}&
      A_2& =a_2\otimes\begin{bmatrix}0&I_\Sc{W}\\I_\Sc{W}&0\end{bmatrix}&\\
      B_1& =I_{\Sc{V}}\otimes\begin{bmatrix}b_1&0\\0&b_2\end{bmatrix}&
      B_2& =I_{\Sc{V}}\otimes\begin{bmatrix}b_2&0\\0&b_1\end{bmatrix}.
   \end{align*}
   Then $[A_1,A_2]$ and $[B_1,B_2]$ define a finite dimensional, fully coisometric representation of $\B{F}_\theta^+$.

   Let $\Sc{V}=\Sc{W}=\B{C}^2$ and let
   \begin{align*}
    a_1&=\begin{bmatrix}1&0\\0&\frac{1}{\sqrt{2}}\end{bmatrix}&
    a_2&=\begin{bmatrix}0&0\\\frac{1}{2}&\frac{1}{2}\end{bmatrix}\\
    b_1&=\begin{bmatrix}0&1\\1&0\end{bmatrix}&
    b_2&=\begin{bmatrix}0&0\\ 0&0\end{bmatrix}.
   \end{align*}
   Construct $[A_1,A_2]$ and $[B_1,B_2]$ as above. Let $\Sc{U}$ be the minimal cyclic coinvariant subspace for the row-contraction $[A_1,A_2]$. A calculation shows that
   \begin{equation*}
    \Sc{U}=\spa\{e_1,e_3,e_5,e_7\},
   \end{equation*}
   where $\{e_1,\ldots,e_8\}$ is the standard orthonormal basis for $\B{C}^8$. However, we have that $B_1^*e_1=e_2\not\in\Sc{U}$, and so $\Sc{U}$ is not the minimal cyclic coinvariant subspace for the representation of $\B{F}_\theta^+$ defined by $[A_1,A_2]$ and $[B_1,B_2]$. In fact, the minimal cyclic coinvariant subspace for this representation is all of $\B{C}^8$. This example shows that atomic representations are special in not needing $\mathbf{m}\geq(1,1)$ in order to satisfy Theorem \ref{prod sys summary}. It is not true of all representations single vertex $2$-graphs.
  \end{example}

  The construction above works because the permutation $\theta$ is very simple. Precisely, if we fix $i$, $\theta$ satisfies $\theta(i,j)=\theta(i,j')$, i.e. $i$ is not changed. Similar constructions of fully coisometric representations of $2$-graphs will work for any $2$-graph defined by a permutation satisfying this condition. These representations will be doubly commuting.

  A general method of constructing  finite dimensional, fully coisometric representations of $2$-graphs which are not partially isometric has proved hard to find. We give below an example of finite dimensional, fully coisometric representation of a $2$-graph which is not doubly commuting.

  \begin{example}\label{not doubly commuting}
   Let $[A_1,A_2]$ and $[B_1,B_2,B_3]$ be row-contractions on $\B{C}^3$ with
   \begin{align*}
    A_1&=\begin{bmatrix}0&0&0\\0&0&0\\ \frac{1}{2}&\frac{1}{2}&0\end{bmatrix}&
    A_2&=\begin{bmatrix}1&0&0\\0&1&0\\ 0&0&\frac{1}{\sqrt{2}}\end{bmatrix}
   \end{align*}
   and
   \begin{align*}
    B_1&=\begin{bmatrix}\frac{1}{2}&\frac{1}{2}&0\\ \frac{1}{2}&\frac{1}{2}&0\\0&0&0\end{bmatrix}&
    B_2&=\begin{bmatrix}\frac{1}{2}&\frac{1}{2}&0\\ -\frac{1}{2}&-\frac{1}{2}&0\\0&0&\frac{1}{\sqrt{2}}\end{bmatrix}&
    B_3&=\begin{bmatrix}0&0&0\\0&0&0\\ \frac{1}{2}& \frac{1}{2}&0\end{bmatrix}.
   \end{align*}
   Then $[A_1,A_2]$ and $[B_1,B_2,B_3]$ define a fully coisometric representation of $\B{F}_\theta^+$ on $\B{C}^3$ where $\theta\in S_{2\times3}$ is the cycle
   \begin{equation*}
    ((1,1),(2,3),(1,2),(1,3)).
   \end{equation*}
   This fully coisometric representation is not doubly commuting.

	It is not hard to see that the minimal cyclic coinvariant space for this representation is $\B{C}^2=\spa\{e_1,e_2\}$, where $\{e_1,e_2,e_3\}$ is the standard orthonormal basis for $\B{C}^3$.
  \end{example}

\end{document}